\documentclass[11pt]{article}

\usepackage[T1]{fontenc}
\usepackage{lmodern}

\usepackage{amsmath, amssymb, amsthm}
\usepackage[margin=2.5cm]{geometry}
\usepackage{color, graphicx}
\usepackage{subfig}
\usepackage{adjustbox}
\usepackage{hyperref}

\usepackage{resizegather}

\usepackage{xargs}
\usepackage[colorinlistoftodos,prependcaption,textsize=footnotesize]{todonotes}
\newcommandx{\attn}[2][1=]{\todo[linecolor=red,backgroundcolor=blue!25,bordercolor=red,#1]{#2}}
\newcommandx{\other}[2][1=]{\todo[linecolor=OliveGreen,backgroundcolor=OliveGreen!25,bordercolor=OliveGreen,#1]{#2}}
\newcommandx{\thiswillnotshow}[2][1=]{\todo[disable,#1]{#2}}

\usepackage{etoolbox}
\makeatletter
\patchcmd{\@maketitle}{\LARGE \@title}{\LARGE\bfseries\@title}{}{}

\renewcommand{\@seccntformat}[1]{\csname the#1\endcsname.\quad}
\makeatother

\definecolor{myviolet}{rgb}{.7,0,1}
\definecolor{darkblue}{rgb}{0,0,.5}

\usepackage{hyperref}
\hypersetup{
	colorlinks=true,        
	linkcolor=darkblue,         
	citecolor=darkblue,         
	urlcolor=darkblue          
}

\makeatletter
\def\th@plain{%
	\thm@notefont{}
	\itshape 
}
\def\th@definition{%
	\thm@notefont{}
	\normalfont 
}

\renewenvironment{proof}[1][\proofname]{\par
	\normalfont
	\topsep0\p@\@plus3\p@ \trivlist
	\item[\hskip\labelsep\itshape
	#1\@addpunct{.}]\ignorespaces
}{%
	\qed\endtrivlist
}
\makeatother

\newtheorem{theorem}{Theorem}[section]
\newtheorem{lemma}{Lemma}[section]
\newtheorem{corollary}{Corollary}[section]
\newtheorem{proposition}{Proposition}[section]
\theoremstyle{definition}
\newtheorem{definition}{Definition}[section]
\theoremstyle{definition}
\newtheorem{example}{Example}[section]
\theoremstyle{definition}
\newtheorem{remark}{Remark}[section]

\usepackage[shortlabels]{enumitem}
\setlist[enumerate]{nosep}

\allowdisplaybreaks

\newcommand{\scal}[2]{\left\langle {#1},{#2} \right\rangle}

\newcommand{\RR}{\ensuremath{\mathbb R}}
\newcommand{\RP}{\ensuremath{\mathbb{R}_+}}

\newcommand{\argmin}{\ensuremath{\operatorname*{argmin}}}

\newcommand{\inte}{\ensuremath{\operatorname{int}}}
\newcommand{\ran}{\ensuremath{\operatorname{ran}}}
\newcommand{\dom}{\ensuremath{\operatorname{dom}}}

\newcommand{\epi}{\ensuremath{\operatorname{epi}}}

\newcommand{\Id}{\ensuremath{\operatorname{Id}}}
\newcommand{\ent}{\ensuremath{\operatorname{ent}}}
\newcommand{\prox}{\ensuremath{\operatorname{Prox}}}

\newcommand{\lev}[1]{\ensuremath{\mathrm{lev}_{\leq #1}\:}}
\newcommand{\fenv}[2][]%
{\ensuremath{\,\overrightarrow{\operatorname{env}}_{#2}}^{#1}}
\newcommand{\benv}[2][]%
{\ensuremath{\,\overleftarrow{\operatorname{env}}_{#2}^{#1}}}
\newcommand{\env}[2][]%
{\ensuremath{{\operatorname{env}}_{#2}^{#1}}}
\newcommand{\W}{\mathcal W}
\newcommand{\bP}[2][]%
{\ensuremath{\,\overleftarrow{\operatorname{P}}_{#2}^{#1}}}

\newcommand{\fP}[2][]%
{\ensuremath{\,\overrightarrow{\operatorname{P}}_{#2}^{#1}}}

\newcommand{\fprox}[1]{\overrightarrow{\operatorname{P}}_%
	{\negthinspace\negthinspace#1}}

\newcommand{\fproxs}[1]{\overrightarrow{\operatorname{s}}_%
	{\negthinspace\negthinspace#1}}

\newcommand{\bprox}[1]{\overleftarrow{\operatorname{P}}_%
	{\negthinspace\negthinspace #1}}

\newcommand{\bproxs}[1]{\overleftarrow{\operatorname{s}}_%
	{\negthinspace\negthinspace #1}}

\usepackage{xcolor}
\usepackage{tikz}
\usetikzlibrary{calc}
\usetikzlibrary{decorations.pathmorphing}

\newcommand{\trimplot}[1]{%
	\begin{adjustbox}{trim={0.1\width} {0.16\height} {0.2\width} {0.16\height},clip=true}
		\includegraphics[width=.4\textwidth]{#1}
	\end{adjustbox}%
}

\begin{document}

\title{\sf Generalized Bregman envelopes and proximity operators}

\author{
Regina S.\ Burachik\thanks{
Mathematics, UniSA STEM, University of South Australia, Mawson Lakes, SA 5095, Australia.
E-mail: \texttt{regina.burachik@unisa.edu.au}.},~
Minh N.\ Dao\thanks{
School of Engineering, Information Technology and Physical Sciences, Federation University Australia, Ballarat, VIC 3353, Australia.
E-mail: \texttt{m.dao@federation.edu.au}.},~~and~
Scott B.\ Lindstrom\thanks{
Department of Applied Mathematics, Hong Kong Polytechnic University, Hong Kong.
E-mail: \texttt{scott.lindstrom@curtin.edu.au}.}
}

\date{}

\maketitle

\begin{abstract} \noindent
Every maximally monotone operator can be associated with a family of convex functions, called the \emph{Fitzpatrick family} or \emph{family of representative functions}. Surprisingly, in 2017, Burachik and Mart\'inez-Legaz showed that the well-known Bregman distance is a particular case of a general family of distances, each one induced by a specific maximally monotone operator and a specific choice of one of its representative functions. For the family of generalized Bregman distances, sufficient conditions for convexity, coercivity, and supercoercivity have recently been furnished. Motivated by these advances, we introduce in the present paper the generalized left and right envelopes and proximity operators, and we provide asymptotic results for parameters. Certain results extend readily from the more specific Bregman context, while others only extend for certain generalized cases. To illustrate, we construct examples from the Bregman generalizing case, together with the natural ``extreme'' cases that highlight the importance of which generalized Bregman distance is chosen.
\end{abstract}

{\small 
\noindent
{\bfseries 2020 Mathematics Subject Classification:}
Primary 90C25;
Secondary  26A51, 26B25, 47H05, 47H09.

\noindent
{\bfseries Keywords:}
convex function,
Fitzpatrick function,
generalized Bregman distance,
maximally monotone operator, 
Moreau envelope,
proximity operator,
regularization,
representative function.
}

\section{Introduction}

In this paper, unless stated otherwise, $\left( X,\left\Vert \cdot \right\Vert \right) $ is a reflexive Banach space with dual $\left( X^\ast,\left\Vert \cdot \right\Vert_\ast \right)$, and $\Gamma_0(X)$ is the set of all proper lower semicontinuous convex functions from $X$ to $\left]-\infty, +\infty\right]$. 

In 1962, Moreau \cite{Mor62} introduced what has come to be known as the Moreau envelope,
\begin{equation}
\env{\gamma,\theta}\colon X\to \left[-\infty, +\infty\right] \colon
y\mapsto \inf_{x\in X} \left \{ \theta(x) +\frac{1}{\gamma}\mathcal{D}_{\|\cdot \|^2/2}(x,y) \right \}
\end{equation}
and its corresponding proximity operator 
\begin{equation}\label{prox-energy}
\prox_{\gamma,\theta}\colon X\to X \colon
y\mapsto \argmin_{x\in X} \left \{\theta(x) +\frac{1}{\gamma}\mathcal{D}_{\|\cdot \|^2/2}(x,y) \right \},
\end{equation}
where $\mathcal{D}_{\|\cdot \|^2/2}\colon (x,y)\mapsto \|x-y\|^2/2$. Moreau worked in a Hilbert space and with parameter $\gamma=1$, and then, Attouch \cite{Attouch77,Attouch84} introduced the more general parameter $\gamma \in \left[0,+\infty\right[$; see also \cite[Chapter~12]{BC17}. 

In 1967, Bregman \cite{Bre67} introduced the distance associated with a differentiable convex function $f$, 
\begin{equation} 
\label{e:Df}
\mathcal{D}_f \colon X \times X \to \left[0, +\infty\right] \colon (x,y) \mapsto
\begin{cases}
f(x) -f(y) -\scal{\nabla f(y)}{x -y} &\text{~if~} y\in \inte \dom f, \\
+\infty &\text{~otherwise},
\end{cases}
\end{equation} 
which now bears his name. When $f$ is the \emph{energy}, namely $f =\|\cdot\|^2/2$, it is clear that, for all $x,y\in X$, $D_f(x,y) =\|x-y\|^2/2$ is the Euclidean distance squared. When $f$ is not the energy, the distance may fail to be symmetric and so one is led to consider the left and right envelopes defined by 
\begin{subequations}\label{d:bregman_envelopes}
\begin{align}
\label{e:bba}
\benv{\gamma,\theta} \colon X\to \left[-\infty, +\infty\right] \colon
y&\mapsto \inf_{x\in X} \left \{\theta(x) +\frac{1}{\gamma}\mathcal{D}_f(x,y) \right \}\\
\label{e:fba} \text{and}\quad
\fenv{\gamma,\theta} \colon X \to \left[-\infty, +\infty\right] \colon
x &\mapsto
\inf_{y\in X} \left \{\theta(y) +\frac{1}{\gamma}\mathcal{D}_f(x,y) \right \},
\end{align}
\end{subequations}
where the left and right proximity operators are defined as in \eqref{prox-energy}, with $\mathcal{D}_f$ in place of $\mathcal{D}_{\| \cdot \|^2/2}$. 

The asymptotic properties of the envelopes and proximity operators for differentiable $f$ with respect to the parameter $\gamma$ were explored in \cite{BDL17}. Bregman distances admit proximal point methods while also casting light on those constructed from the classical Moreau envelopes; see, for example, \cite{BB97,BCN06,ReginaBreg2010,ReginaBreg98,BurKas12,BC01,CZ92,CT93,Eck93,Kiw97}, as well as \cite[Chapter~6]{ReginaBook2008}. 

Recently, Burachik and Mart\'inez-Legaz \cite{BM-L18} have introduced two distances based on a representative function $h$ of a maximally monotone operator $T$:
\begin{subequations}
\begin{align}
\mathcal{D}_T^{\flat,h}(x,y) &:=\inf_{v\in Ty}\left( h(x,v)-\langle
x,v\rangle \right)\\ 
\text{and}\quad \mathcal{D}_T^{\sharp,h}(x,y) &:=\sup_{v\in Ty}\left( h(x,v)-\langle
x,v\rangle \right).
\end{align}
\end{subequations}
When $T =\nabla f$ and $h(x,\nabla f(y)) =(f \oplus f^\ast)(x,\nabla f(y)) :=f(x)+f^\ast(\nabla f(y))$ (which is the Fenchel--Young representative), these distances reduce (under mild domain conditions), to the Bregman distance $\mathcal{D}_f(x,y)$. We therefore call these more general distances as the \emph{generalized Bregman distances} (or \emph{GBDs}).

Bregman proximal methods may be desirable for problems of high dimension. For certain objectives, using a different Bregman distance than the square norm reduces the computational complexity of computing a proximal update. For an example of reduction from $n^3$ to $n^2$ with SDP-representable constraints, see \cite{chao2018entropic,jiang2021bregman}). Their work illustrates that a well selected Bregman distance should complement the problem structure. To that end, choices from the broader class of GBDs may turn out to be useful.

In \cite{chretien2016bregman}, the authors devise a prox-Bregman technique for feature extraction in machine learning, with an application to gene expression problems, wherefore outliers are present while data is noisy and sometimes missing. The latter properties are known to prevent other approaches---such as principal component analysis, cluster analysis, and polytomic logistic regression---from delivering useful results. In particular, the Bregman distance is built from the Boltzmann--Shannon entropy, because this choice naturally enforces nonnegativity constraints. This distance is equal almost everywhere to the GBD for the Fenchel--Young representative of the logarithm; we study here the broader GBD family for the logarithm, noting that more choices of Bregman distances may result in more efficient methodologies. 

Bregman distances also have a prominent role in the solution of important variational  problems. An example of this is its use for solving the equilibrium problem in Banach spaces, see, e.g., \cite{djafari2020proximal} as well as \cite{BurKas12}. This is one reason we choose to work in the more general Banach space setting.

More recently, \cite{BDL19} has provided a framework of sufficient conditions for coercivity and supercoercivity of the left and right GBDs. It has also shown how such properties are useful for establishing coercivity of the sum of the distance together with a function in $\Gamma_0(X)$. 

Hence, it is natural to use these new coercivity properties for carrying out a detailed analysis of the envelopes and proximity operators that are obtained when the GBD replaces the Bregman distance in \eqref{e:bba} and \eqref{e:fba}. In particular, ours can be seen as a unifying analysis which includes Moreau envelopes as a particular case. The goal of the present work is to furnish this analysis. We characterize the domains of the envelopes. We then show under what conditions the GBD envelopes possess the same advantageous properties that we have when specializing to Bregman case, and under what conditions those properties may be lost. We will illustrate with the same GBDs used in \cite{BDL19}. We show, in particular, that the GBDs that arise from using the Fitzpatrick representative generically share the same desirable properties as their Bregman-generalizing counterparts. Such results will be important if some Fitzpatrick cases possess computational advantages over their Bregman counterparts or admit envelopes that are useful for examining existing algorithms through their dual characterizations.

\subsection*{Outline and contributions}

In Section~\ref{sec:preliminaries}, we recall the generalized Bregman distances, along with some of their basic properties. We also recall the coercivity framework as well as the computed distances recently established in \cite{BDL19}. Moreover, we explain why these distances are important, since we use them to build the envelopes and proximity operators in our examples. 

In Section~\ref{sec:envelopes}, we introduce the left and right GBD envelopes and their associated proximity operators. We characterize the domains of the envelopes, and we provide sufficient conditions to guarantee the attainment of minimizers so that the proximity operators have nonempty images. The sufficient conditions rely upon the framework for coercivity established in \cite{BDL19}. 

In Section~\ref{sec:asymptotic}, we provide asymptotic results for the parameter $\gamma$. We then show how the results in the setting of GBDs vary from those we obtain more easily when specializing to Bregman distances. We illustrate all examples with both left and right versions, along with images of the envelope nets for a selection of $\gamma$ values. For all examples, we include three prototypical cases: the case of the distance constructed from the Fenchel--Young representative (which, under certain conditions, coincides with the classical Bregman distance), as well as the distances constructed from the smallest and largest members of the representative function set.

We conclude in Section~\ref{sec:conclusion}, and we provide explicit forms for all of our computed examples and proximity operators in Appendix~\ref{sec:appendix}.

\section{Preliminaries}\label{sec:preliminaries}

Given a nonempty subset $C$ of $X$, we denote by $\iota_C\colon X\to \left]-\infty, +\infty\right]$ the \emph{indicator function} of $C$, i.e., $\iota_C(x) :=0$ when $x\in C$ and $\iota_C(x):=+\infty$ otherwise. We will denote by ${\rm int}(C)$ the interior of $C$ and by $\overline{C}$ the closure of $C$. 

Let $f\colon X\to \left]-\infty, +\infty\right]$. The \emph{domain} of $f$ is defined by $\dom f :=\{x\in X: f(x) <+\infty\}$, the \emph{lower level set} of $f$ at height $\xi\in \mathbb{R}$ by $\lev{\xi} f :=\{x\in X: f(x)\leq \xi\}$, and the \emph{epigraph} of $f$ by $\epi f :=\{(x,\rho)\in X\times \mathbb{R}: f(x)\leq \rho\}$. We say that $f$ is \emph{proper} if $\dom f\neq \varnothing$ and \emph{lower semicontinuous (lsc)} at $\bar{x}$ if $f(\bar{x})\leq \liminf_{x\to \bar{x}} f(x)$. Unless specifically mentioned, these concepts are with respect to the strong (norm) topology. The function $f$ is said to be \emph{convex} if  
\begin{equation}
\forall x,y\in X,\ \forall \lambda\in [0,1],\quad 
f((1-\lambda)x+\lambda y)\leq (1-\lambda)f(x) +\lambda f(y);
\end{equation}
\emph{coercive} if $\lim_{\|x\|\to +\infty} f(x) =+\infty$; and \emph{supercoercive} if $\lim_{\|x\|\to +\infty} f(x)/\|x\| =+\infty$. 

For a proper function $f\colon X\to \left]-\infty, +\infty\right]$, its \emph{$\varepsilon$-subdifferential} ($\varepsilon\in \RP$) is the point-to-set mapping $\partial_\varepsilon f\colon X\rightrightarrows X^{\ast }$ given by 
\begin{equation}\label{e:partialf}
\partial_\varepsilon f(x) :=\{v\in X^\ast: \forall y\in X,\; \langle y-x,v \rangle \leq f(y) -f(x) +\varepsilon\},
\end{equation}
its \emph{subdifferential} is $\partial f :=\partial_0 f$, and its \emph{Fenchel conjugate} is the function $f^\ast\colon X^\ast\to \left]-\infty, +\infty\right]$ given by
\begin{equation}
f^\ast(v) :=\sup_{x\in X}\{\langle x,v\rangle -f(x)\}.
\end{equation} 
From the definition, we directly obtain the \emph{Fenchel--Young inequality}
\begin{equation}
\forall (x,v)\in X\times X^\ast,\quad f(x) +f^\ast(v)\geq \langle x,v \rangle
\end{equation}
and for $\varepsilon\in \RP$, we also have the following well-known characterization of $\partial_\varepsilon f$: 
\begin{equation}\label{FY}
f(x) +f^\ast(v) \leq \langle x,v \rangle +\varepsilon \iff v\in \partial_\varepsilon f(x).
\end{equation}

Given a point-to-set operator $T\colon X\rightrightarrows X^{\ast}$, its \emph{domain} is $\dom T :=\{x\in X: Tx\neq \varnothing\}$, its \emph{range} is $\ran T :=T(X)$, and its \emph{graph} is $\mathcal{G}(T) :=\{(x,x^{\ast})\in X\times
X^{\ast}: x^{\ast}\in Tx\}$. We say that $T$ is \emph{maximally monotone} if 
\begin{equation}\label{d:maximallymonotone}
(x,u) \in \mathcal{G}(T) \quad \iff \quad \forall (y,v)\in \mathcal{G}(T),\quad \langle x-y,u-v \rangle \geq 0.
\end{equation}
More properties and facts on maximally monotone operators can be found in \cite{BC17,ReginaBook2008}.

\subsection{Representative functions}

Let $S\colon X\rightrightarrows X^\ast$ be a maximally monotone operator. Following \cite[Definition~2.3]{BM-L18}, we say that a function $h\colon X\times X^\ast\to \left]-\infty, +\infty\right]$ \emph{represents $S$} if it satisfies the following conditions:
\begin{enumerate}[label=(\alph*)]
\item
$h$ is convex and norm $\times $ weak$^{\ast }$ lower semicontinuous in $X\times X^\ast$ (the weak$^{\ast }$ topology in $X^\ast$ is the smallest topology that makes continuous the linear functionals induced by $x\in X$).
\item
$\forall (x,v)\in X\times X^\ast,\ h(x,v)\geq \langle x,v\rangle$.
\item\label{h:c}
$h(x,v)=\langle x,v\rangle \Longleftrightarrow (x,v)\in \mathcal{G}(S)$.
\end{enumerate}
In this situation, we denote $h\in \mathcal{H}(S)$ and call $\mathcal{H}(S)$ the \emph{Fitzpatrick family of $S$}. We will make use, in particular, of three prototypical members of $\mathcal{H}(S)$. These are as follows.
\begin{enumerate}
\item 
The \emph{Fitzpatrick function of $S$}, denoted as $F_S$, is defined by
\begin{equation}
F_{S}(x,y) :=\sup_{(z,w) \in \mathcal{G}(S)} \left( \langle z-x,y-w \rangle +\langle x,y\rangle \right).    
\end{equation}
It is well known that $F_S$ is the smallest member of $\mathcal{H}(S)$ in the sense that $F_S\leq h$ for any $h\in \mathcal{H}(S)$, see \cite{Fitzpatrick 1988};
\item 
We denote as $\sigma_{S}$ the largest member of $\mathcal{H}(S)$ in the sense that $\sigma_S\ge h$ for any $h\in \mathcal{H}(S)$;
\item 
In the case when $T=\partial f$ for $f\in \Gamma_0(X)$, we also consider 
the \emph{Fenchel--Young representative}, denoted as $f\oplus f^\ast \in \mathcal{H}(\partial f)$, where $f\oplus f^\ast\colon X\times X^\ast\to \left]-\infty, +\infty\right]$ is given by
\begin{equation}
\forall (x,v)\in X\times X^\ast,\quad (f\oplus f^\ast)(x,v) :=f(x) +f^\ast(v).
\end{equation}
\end{enumerate}

\subsection{A distance between point-to-set operators}\label{ss:generalizesbregman}

From now on, we assume that $S\colon X\rightrightarrows X^\ast$ is a maximally monotone operator, $h\in \mathcal{H}(S)$, and $T\colon X\rightrightarrows X^\ast$ any point-to-set operator. As in \cite[Definition~3.1]{BM-L18}, for each $(x,y)\in X\times X$, we define 
\begin{subequations}\label{D}
\begin{align}
\mathcal{D}_T^{\flat,h}(x,y) &:=\begin{cases}
\inf\limits_{v\in Ty} \left( h(x,v)-\langle
x,v\rangle \right) & \text{if~} (x,y) \in \dom S\times \dom T, \\
+\infty & \text{otherwise}
\end{cases} \label{Dsharp}\\ 
\text{and}\quad \mathcal{D}_T^{\sharp,h}(x,y) &:=\begin{cases} 
\sup\limits_{v\in Ty} \left( h(x,v)-\langle
x,v\rangle \right) & \text{if~} (x,y) \in \dom S\times \dom T, \\
+\infty & \text{otherwise.}
\end{cases}\label{Dflat}
\end{align}
\end{subequations}
When $T$ is point to point, we simply write $\mathcal{D}_T^h :=\mathcal{D}_T^{\flat,h} =\mathcal{D}_T^{\sharp,h}.$ In some situations, we will refer to both distances \eqref{Dsharp} and \eqref{Dflat} simultaneously by the symbol $\mathcal{D}_{T}^{\star,h}$.

If a distance is of form \eqref{Dsharp} or \eqref{Dflat}, we call it a \emph{generalized Bregman distance} or \emph{GBD} for short. We mentioned before that the GBDs specialize to the Bregman distance under certain circumstances, which we now make precise. To a proper and convex function $f\colon X\to \left]-\infty, +\infty\right]$, we associate two Bregman distances (see \cite{Kiw97}) defined by
\begin{subequations}\label{d:Bregman}
\begin{align}
\mathcal{D}_f^\flat(x,y) :=& f(x)-f(y)+\inf_{v \in \partial f(y)} \langle y-x,v \rangle \label{d:Bregman_flat}\\
\text{and}\quad \mathcal{D}_f^\sharp(x,y) :=& f(x)-f(y)+\sup_{v \in \partial f(y)} \langle y-x,v \rangle.\label{d:Bregman_sharp}
\end{align}
\end{subequations}  
It is known that the GBDs specialize to the Bregman distances in the case where the Fenchel--Young representative is used \cite[Proposition~3.5]{BM-L18}, under mild domain conditions illuminated in \cite[Proposition~2.2]{BDL19}. We recall this result in the following proposition. 

\begin{proposition}[GBDs specialize to Bregman distances]
\label{p:fitz_bregman}
Let $f\in \Gamma_0(X)$. Then, for all $(x,y)\notin (\dom f\setminus \dom\partial f)\times \dom\partial f$,
\begin{equation}
\mathcal{D}_{\partial f}^{\flat, f \oplus f^{\ast}}(x,y) =\mathcal{D}_f^{\flat}(x,y) \quad\text{and}\quad
\mathcal{D}_{\partial f}^{\sharp, f \oplus f^{\ast}}(x,y) =\mathcal{D}_f^{\sharp}(x,y).
\end{equation}
\end{proposition}

\begin{remark}
By Proposition~\ref{p:fitz_bregman}, we see that, in the case when $\dom f \setminus \dom \partial f = \varnothing$, the two kinds of distances are everywhere equal. However, if $f$ is the Boltzmann--Shannon (see \eqref{def:entropy}), then $\dom f \setminus \dom \partial f = \{0\}$, and the two kinds of distances fail to be equal on the set $\{(0,y) | \; y>0 \}$ (see \cite{BDL19} for more details).
\end{remark}

As a motivation for considering these new distances, we provide below connections between our generalized distances and solutions of variational problems. First, we consider the problem of finding zeros of a sum of operators, and second, the problem of minimizing a DC (difference of convex) function. In both cases, we need $S$ different than $T$, with an operator $T$ which may not be monotone.

\begin{remark}\label{ES} 
In the next proposition, we use the Eberlein--\u Smulian theorem, which states that a subset of a Banach space is weakly compact if and only if it is weakly sequentially compact (see \cite[Chapter~III, page 18]{Diestel}). We also use the fact that, if $X$ is a reflexive Banach space and a map $V:X\rightrightarrows X^*$ is locally bounded at a point in the interior of its domain, then there is a neighbourhood of that reference point which is norm-closed and bounded, and hence weakly compact (by Bourbaki--Alaoglu's theorem and reflexivity). We then use the Eberlein--\u Smulian theorem to deduce that the given neighbourhood is in fact weakly sequentially compact. 
\end{remark}

Given $h\in \mathcal{H}(S)$, the \emph{enlargement} $S^h_\varepsilon$ of $S$ is defined by
\begin{equation}\label{eq:SE}
S^h_\varepsilon x :=\{v\in X^\ast: h(x,v)-(x,v)\leq \varepsilon\}.
\end{equation} 
More details on $S^h_\varepsilon$ can be found, e.g., in \cite{ReginaBook2008,bs2002}. 
We will show that the generalized distances can be used to define approximate solutions of problem  
\begin{equation}\label{eq:SP}
\text{find $x\in X$ such that~~} 0\in Sx + Tx.
\end{equation}
The proof of the next result follows closely the one in \cite[Proposition 3.7]{BM-L18}, but we include it here for the convenience of the reader. 

\begin{proposition}
\label{p:sum}
Let $X$ be a reflexive Banach space. Suppose that $S\colon X\rightrightarrows X^\ast$ is a maximally monotone operator and $T\colon X\rightrightarrows X^\ast$ is a point-to-set operator. Fix any $h\in \mathcal{H}(S)$, $\varepsilon\in \RP$, and $x\in X$. Consider the following statements.
\begin{enumerate}[label =(\alph*)]
\item\label{p:sum_approx} 
$0\in S^h_\varepsilon x + Tx$.
\item\label{p:sum_dist} 
$\mathcal{D}_{-T}^{\flat,h}(x,x)\leq \varepsilon$.
\end{enumerate}
Then \ref{p:sum_approx} $\implies$ \ref{p:sum_dist}. Moreover, if $\dom T$ is open and $T$ is locally bounded with weakly closed images, then the two statements are equivalent. 
\end{proposition}
\begin{proof}
\ref{p:sum_approx} $\implies$ \ref{p:sum_dist}: Note that \ref{p:sum_approx} implies the existence of $w\in (-Tx)\cap S^h_\varepsilon x$. By definition of $\mathcal{D}_{-T}^{\flat,h}$, we have
\begin{equation}
\mathcal{D}_{-T}^{\flat,h}(x,x) =\inf_{v\in -Tx} h(x,v)-\langle x,v \rangle \leq h(x,w)-\langle x,w \rangle \leq \varepsilon,    
\end{equation}
where we used the fact that $w\in S^h_\varepsilon x$ (see \eqref{eq:SE}) in the rightmost inequality. 

\ref{p:sum_dist} $\implies$ \ref{p:sum_approx}: It follows from \ref{p:sum_dist} that $\mathcal{D}_{-T}^{\flat,h}(x,x) <+\infty$, and so $x\in \dom (-T) =\dom T$. Now, assume that $\dom T$ is open and $T$ is locally bounded with weakly closed images. Then, the same properties hold for $-T$ too. Consequently, the set $-Tx$ is contained in a neighbourhood which is norm-closed and bounded. Using now Remark \ref{ES}, we have that the set $-Tx$ is contained in a neighbourhood which is weakly sequentially compact. Using also the assumption that $h(x,\cdot)$ is weakly continuous, we conclude that the infimum in the expression for $\mathcal{D}_{-T}^{\flat,h}(x,x)$ must be attained at some point in $-Tx$. Namely, there exists $v\in -Tx$ such that
\begin{equation}
h(x,v)-\langle x,v \rangle =\mathcal{D}_{-T}^{\flat,h}(x,x) \leq \varepsilon,   
\end{equation}
which implies that $v\in S^h_\varepsilon x$. Since $-v\in Tx$, we get \ref{p:sum_approx}.
\end{proof}

Note that, if $x$ satisfies condition \ref{p:sum_approx} in Proposition~\ref{p:sum}, it can be seen as an approximate solution of \eqref{eq:SP}. Indeed, by taking $\varepsilon=0$, condition \ref{p:sum_approx} becomes \eqref{eq:SP}. In the latter case (i.e., when $\varepsilon=0$), condition \ref{p:sum_dist} in Proposition~\ref{p:sum} can be seen as a necessary optimality condition for problem \eqref{eq:SP}. The condition becomes sufficient when $T$ verifies additional hypotheses. 

We show next that an optimality condition for minimizing a DC function can be expressed by means of the sharp distance. In the proposition below, the equivalence between statements \ref{p:DC_min} and \ref{p:DC_oc} is well known in finite dimensional spaces (see, e.g. \cite[Theorem~3.1]{HU95}). The analogous result in Banach spaces is hard to track down, so we decided to include its proof here.

\begin{proposition}
\label{p:DC}
Let $f\colon X\to \left]-\infty, +\infty\right]$ and $g\colon X\to \RR$ be proper lower semicontinuous convex functions. Then, the following statements are equivalent
\begin{enumerate}[label =(\alph*)]
\item\label{p:DC_min}
$x$ is a global minimum of $f-g$ on $X$.
\item\label{p:DC_oc} 
For all $\varepsilon\in \RP$, $\partial_{\varepsilon}g(x)\subseteq \partial_{\varepsilon}f(x)$.
\item\label{p:DC_dist}
For all $\varepsilon\in \RP$, $\mathcal{D}_{\partial_{\varepsilon}g}^{\sharp,f\oplus f^*}(x,x)\leq \varepsilon$.
\end{enumerate}
\end{proposition}
\begin{proof}  
\ref{p:DC_min} $\implies$ \ref{p:DC_oc}: Assume that $x$ is a global minimum of $f-g$ on $X$. Then, for all $y\in X$, $f(x) -g(x)\leq f(y) -g(y)$, and so $g(y) -g(x)\leq f(y) -f(x)$. Now, let any $\varepsilon\in \RP$ and any $v\in \partial_\varepsilon g(x)$. We have, for all $y\in X$, 
\begin{equation}
\langle y-x,v \rangle +\varepsilon\leq g(y) -g(x)\leq f(y) -f(x),    
\end{equation}
which implies that $v\in \partial_\varepsilon f(x)$. Therefore, $\partial_\varepsilon g(x)\subseteq \partial_\varepsilon f(x)$.

\ref{p:DC_oc} $\implies$ \ref{p:DC_min}: Assume that $x$ is not a global minimum of $f-g$ on $X$. That is, there exists $\overline{x}\in X$ such that $f(\overline{x}) -g(\overline{x}) <f(x) -g(x)$. Then, we can pick $0 <\varepsilon_0 < [g(\overline{x}) -g(x)] -[f(\overline{x}) -f(x)]$ and $u\in \partial_{\varepsilon_0} g(\overline{x})$ (see \cite[Theorem~2.4.4(iii)]{Zal02}). By definition, for all $y\in X$, 
\begin{equation}
\langle y-\overline{x},u \rangle\leq g(y) -g(\overline{x}) +\varepsilon_0,
\end{equation}
which can be written as
\begin{equation}
\langle y-x,u \rangle \leq g(y) -g(x) +\varepsilon, 
\end{equation}
where $\varepsilon :=g(x) -g(\overline{x}) +\varepsilon_0 -\langle x-\overline{x},u \rangle\geq 0$. Since $y\in X$ is arbitrary, the inequality above yields $u\in \partial_\varepsilon g(x)$. On the other hand, 
\begin{equation}
\langle \overline{x}-x,u \rangle =g(\overline{x}) -g(x) -\varepsilon_0 +\varepsilon >f(\overline{x}) -f(x) +\varepsilon,    
\end{equation}
which implies that $u\notin \partial_\varepsilon f(x)$. Hence, \ref{p:DC_oc} does not hold. We deduce that if \ref{p:DC_oc} holds, then so does \ref{p:DC_min}. 

We have shown the equivalence between \ref{p:DC_min} and \ref{p:DC_oc}. To see the one between \ref{p:DC_oc} and \ref{p:DC_dist}, we note that 
\begin{equation}
\mathcal{D}_{\partial_{\varepsilon}g}^{\sharp,f\oplus f^*}(x,x) =\sup_{v\in \partial_{\varepsilon}g(x)} \left( (f\oplus f^\ast)(x,v) -\langle x,v \rangle \right) =\sup_{v\in \partial_{\varepsilon}g(x)} \left( f(x) +f^\ast(v) -\langle x,v \rangle \right).    
\end{equation}
Fix any $\varepsilon\in \RP$, the above expression and characterization \eqref{FY} of $\partial_{\varepsilon}f$ imply that 
\begin{subequations}
\begin{align}
\mathcal{D}_{\partial_{\varepsilon}g}^{\sharp,f\oplus f^*}(x,x)\leq \varepsilon &\iff \forall v\in \partial_{\varepsilon}g(x),\quad f(x) +f^\ast(v) -\langle x,v \rangle \leq \varepsilon \\
&\iff \forall v\in \partial_{\varepsilon}g(x),\quad v\in \partial_{\varepsilon}f(x) \\
&\iff \partial_{\varepsilon}g(x)\subseteq \partial_{\varepsilon}f(x),
\end{align}
\end{subequations}
which completes the proof.
\end{proof}

We will make use of the following lemma to simplify our analysis of the asymptotic behaviour in Section~\ref{sec:asymptotic}. 

\begin{lemma}
\label{l:domD}
Let $f\in \Gamma_0(X)$ and suppose that $h\in \mathcal{H}(\partial f)$ satisfies
\begin{equation}\label{eqn:dom}
\forall (x,v)\in \dom f\times \ran\partial f,\quad h(x,v)\leq (f\oplus f^\ast)(x,v) =f(x) +f^\ast(v).
\end{equation}
Then the following hold:
\begin{enumerate}
\item\label{l:domD_flat} 
$\dom \mathcal{D}_{\partial f}^{\flat,h} =\dom\partial f\times \dom\partial f$.
\item\label{l:domD_sharp} 
$\dom\partial f\times \inte(\dom\partial f)\subset	\dom \mathcal{D}_{\partial f}^{\sharp, h} \subseteq \dom\partial f\times \dom\partial f.$ Consequently, when $\dom\partial f$ is open, then $\dom \mathcal{D}_{\partial f}^{\sharp, h} = \dom\partial f\times \dom\partial f$.
\end{enumerate}
\end{lemma}
\begin{proof}
Let $(x,y)\in \dom\partial f\times \dom\partial f$. It follows from the assumption on $h$ and Proposition~\ref{p:fitz_bregman} that 
\begin{equation}\label{eq1}
\mathcal{D}_{\partial f}^{\star,h}(x,y)\leq \mathcal{D}_{\partial f}^{\star,f\oplus f^\ast}(x,y) =\mathcal{D}_f^\star(x,y),
\end{equation}
where $\star\in \{\flat,\sharp\}$.

\ref{l:domD_flat}: By definition, we have that $\dom \mathcal{D}_{\partial f}^{\flat, h}\subseteq \dom\partial f\times \dom\partial f$.  Hence, it is enough to prove the opposite inclusion. Indeed, fix any $v_0\in \partial f(y)$. By \eqref{eq1} for $\star =\flat$ and then by Cauchy--Schwarz inequality, we have that
\begin{subequations}
\begin{align}
\mathcal{D}_{\partial f}^{\flat,h}(x,y) &\leq \mathcal{D}_f^\flat(x,y) =f(x) -f(y) +\inf_{v \in \partial f(y)} \langle y-x,v \rangle \\
&\leq f(x) -f(y) +\|y-x\|\cdot \|v_0\| <+\infty,
\end{align}
\end{subequations}
Here, the final inequality follows from the fact that $x,y\in \dom\partial f\subseteq \dom f$. Thus $(x,y)\in \dom \mathcal{D}_{\partial f}^{\flat,f\oplus f^\ast}$, and we deduce that $\dom \mathcal{D}_{\partial f}^{\flat,h} =\dom\partial f\times \dom\partial f$.

\ref{l:domD_sharp}: As in part \ref{l:domD_flat}, we always have that $\dom \mathcal{D}_{\partial f}^{\sharp, h}\subseteq \dom\partial f\times \dom\partial f$. For the leftmost inclusion, assume that $(x,y)\in \dom\partial f\times \inte(\dom\partial f)$. Since $y\in \inte(\dom\partial f)$ and $\partial f$ is maximally monotone, the set $\partial f(y)$ is bounded, so there exists a constant $M(y)\geq 0$ such that $\|v\|\leq M(y)$ for every $v\in \partial f(y)$. Altogether, \eqref{eq1} for $\star =\sharp$ and Cauchy--Schwarz inequality give
\begin{subequations}
\begin{align}
\mathcal{D}_{\partial f}^{\sharp,h}(x,y) &\leq \mathcal{D}_f^\sharp(x,y) =f(x) -f(y) +\sup_{v \in \partial f(y)} \langle y-x,v \rangle \\
&\leq f(x) -f(y) +\|y-x\|\cdot M(y) <+\infty,
\end{align}
\end{subequations}
because $x,y$ are fixed. This proves both inclusions. The last statement in \ref{l:domD_sharp} follows directly from the first one.
\end{proof}

\begin{remark}
The rightmost inclusion in Lemma~\ref{l:domD}\ref{l:domD_sharp} might be strict, in contrast to part \ref{l:domD_flat}. Let $B :=\{(a,b)\in \mathbb{R}^2: a^2+b^2\leq 1\}$ the unit ball in two dimensions and consider $f :=\iota_{B}$ the indicator function of the unit ball. Take $y=(1,0)\in B$, and any $x=(x_1,x_2)\in B$ such that $x_1<1$ (i.e., any $B\ni x\neq y$). Then $\partial f(y)=\{(t,0): t\geq 0\}$ and $f(x)=f(y)=0$, thus
\begin{equation}
\mathcal{D}_{\partial f}^{\sharp,f\oplus f^\ast}(x,y)=\sup_{t\geq 0} \langle y-x,ty \rangle=  (1-x_1)\sup_{t\geq 0}t=+\infty.    
\end{equation}
Therefore, $(x,y)\not\in \dom \mathcal{D}_{\partial f}^{\sharp,f\oplus f^\ast}$ for all $x\neq y$.
\end{remark}

\subsection{Examples of important generalized Bregman distances}

For our examples in this section, $X =X^\ast =\mathbb{R}$ and, $T =S =\partial f$ is point-to-point on $\inte \dom f$ in which case we simply write $\mathcal{D}^{h}$ for a specific choice of $h\in \mathcal{H}(\partial f)$ (not to be confused with $\mathcal{D}_f$, the classical Bregman distance for $f$). In these cases, we will refer to the representative function used by its name. Specifically,
\begin{enumerate}
\item\label{d:DFitzpatrick} When $h=F_S$ is the Fitzpatrick function for a maximally monotone operator $S$, we will write $\mathcal{D}^{F_{S}}$;
\item When $h=\sigma_{S}$, the largest member of  $\mathcal{H}(S)$, we will write $\mathcal{D}^{\sigma_{S}}$;
\item\label{d:DFenchelYoung} When $h =f\oplus f^\ast \in \mathcal{H}(\partial f)$ for $f \in \Gamma_0(X)$ is the Fenchel--Young representative, we will denote this by $\mathcal{D}^{f\oplus f^\ast}$.
\end{enumerate} 

\begin{remark}[The lower closed distance]
\label{r:lsc}
The function $\mathcal{D}_T^{\star,h}(\cdot,y)$ may not be lower semicontinuous at $x \in \overline{\dom} \partial f \setminus \dom \partial f$. Additionally, for $y \notin \inte \dom(T)$, the distance may \emph{not} be lower semicontinuous with respect to the right variable either. For more details on the semicontinuity properties of the GBDs, see \cite[section 3]{BM-L18}. For these reasons, \cite{BDL19} introduced the \emph{lower closed GBD} $\overline{\mathcal{D}}^{\star,h}_T$ defined by
\begin{align}
\epi \overline{\mathcal{D}}^{\star,h}_T =\overline{\epi} \mathcal{D}^{\star,h}_T,
\end{align}
where, as before, $\star\in \{\flat,\sharp\}$. The lower closed distances $\overline{\mathcal{D}}^{F_S}$, $\overline{\mathcal{D}}^{\sigma_{S}}$, and $\overline{\mathcal{D}}^{f\oplus f^\ast}$ are defined analogously (with ${\mathcal{D}}^{F_S}, {\mathcal{D}}^{\sigma_{S}}$ and ${\mathcal{D}}^{f\oplus f^\ast}$ as in \ref{d:DFitzpatrick}--\ref{d:DFenchelYoung} above). For all our computed examples, in the cases when $\overline{\mathcal{D}}^{h}$ and $\mathcal{D}^{h}$ do not agree, we will compute with $\overline{\mathcal{D}}^{h}$ so as to have a lower semicontinuous distance.
\end{remark}

The authors in \cite{BDL17,BDL19} illustrated their findings with \emph{energy}, whose Bregman distance specializes to the Moreau case and the \emph{Boltzmann--Shannon entropy}, whose importance we will soon recall. Naturally, we will illustrate our results about envelopes and proximity operators using the GBDs associated with these operators that were computed in \cite{BDL19}. They are as follows.

\begin{example}[Energy]
\label{ex:dist:energy}
In the case where $f: x \mapsto \frac{1}{2}x^2$ is the \emph{energy}, we have $\partial f=\Id$. The distances constructed from the smallest and biggest members of $\mathcal{H}(\Id)$ are
\begin{equation}
\mathcal{D}^{F_{\Id}}(x,y) =\frac14(x-y)^2 \quad \text{and}\quad \mathcal{D}^{\sigma_{\Id}} =\iota_{\mathcal{G}(\Id)}. 
\end{equation}
\end{example}
Already from this example, we see that we should not expect all asymptotic properties of Bregman envelopes extend to GBD envelopes, because envelopes associated with the distance $\mathcal{D}^{\sigma_{\Id}}$ will always be vacuously equal to the function being regularized, while their associated proximity operators will vacuously be equal to $\Id$.

\begin{example}[Kullback--Leibler divergence and GBD $\overline{\mathcal{D}}^{\ent \oplus \ent^*}$]
\label{ex:KullbackLiebler}
The (negative) Boltzmann--Shannon entropy is defined as 
\begin{equation}\label{def:entropy}
\ent:\RR \rightarrow \left]-\infty, +\infty\right] \colon x\mapsto \begin{cases} x\log x -x & \text{if~}  x > 0, \\
0 & \text{if~}  x=0,\\
+\infty & \text{otherwise.}\end{cases} 
\end{equation}

The Boltzmann--Shannon entropy is particularly important and natural to consider, because its derivative is $\log$, its conjugate is $\ent^* = \exp$, and its associated Bregman distance is the Kullback--Leibler divergence,
\begin{equation}
\mathcal{D}_{\ent}:(x,y) \mapsto \begin{cases}x(\log(x)-\log(y))-x+y & \text{if~}  y >0,\\
y & \text{if~}  y >0\;\text{and}\; x= 0,\\
+\infty & \text{otherwise},\end{cases}    
\end{equation}
which is frequently used as a measure of distance between positive vectors in information theory, statistics, and portfolio selection. The GBD associated with the Fenchel--Young representative $\ent \oplus \ent^* \in \mathcal{H}(\log)$ is
\begin{equation}
\mathcal{D}^{\ent \oplus \ent^*}:(x,y) \mapsto \begin{cases}x(\log(x)-\log(y))-x+y & \text{if~}  x,y >0,\\
+\infty & \text{otherwise}.\end{cases}    
\end{equation}
Thus, it may be seen that the \emph{Bregman} distance of the \emph{Boltzmann--Shannon entropy} is the special case of the \emph{GBD} for the Fenchel--Young representative of the \emph{logarithm} function, except on the set $\dom f \setminus \dom \partial f \times \dom \partial f = \{0\}\times \left]0,+\infty \right[$ (see Proposition~\ref{p:fitz_bregman} and Remark~\ref{r:lsc}). Its lower closure is given by 
\begin{equation}\label{def:Bregman_lower_closure}
\overline{\mathcal{D}}^{\ent \oplus \ent^*}:(x,y) \mapsto \begin{cases}
x(\log(x)-\log(y))-x+y & \text{if~}  y >0,x\geq 0,\\
0 &\text{if~}  y=x=0,\\
+\infty & \text{otherwise}
\end{cases}
\end{equation}
and is shown in Figure~\ref{D:bregman}. Hence, this new distance allows us to extend the domain of the classical Bregman distance to the boundary points.
\end{example}

\begin{figure}
\begin{center}
\subfloat[{$\overline{\mathcal{D}}^{F_{\log}}$}\label{D:fitz}]{\includegraphics[width=.3\textwidth]{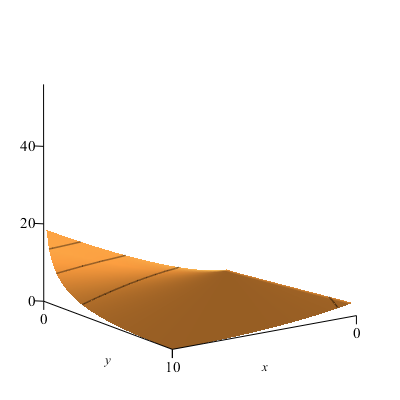}}
\subfloat[{$\overline{\mathcal{D}}^{\ent \oplus \ent^*}$}\label{D:bregman}]{\includegraphics[width=.3\textwidth]{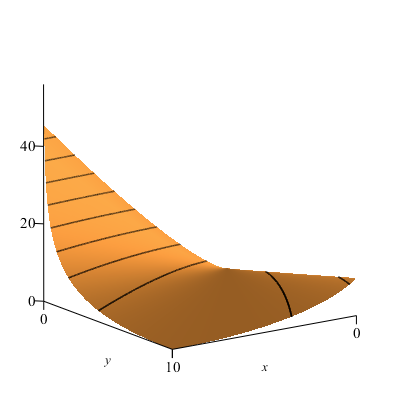}}
\subfloat[{$\overline{\mathcal{D}}^{\sigma_{\log}}$}\label{D:sigma}]{\includegraphics[width=.3\textwidth]{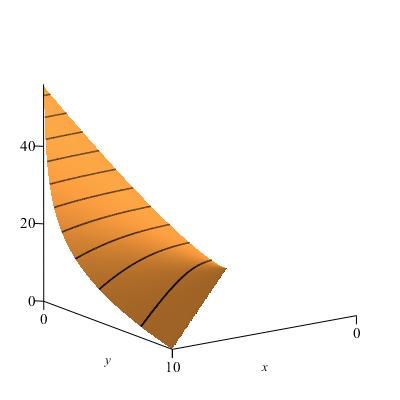}}
\end{center}
\caption{Three prototypical distances constructed from $\mathcal{H}(\log)$.}\label{fig:D}
\end{figure}

\begin{example}[Kullback--Leibler divergence and GBD $\overline{\mathcal{D}}^{F_{\log}}$]
\label{ex:Fitzlog}
The GBD constructed with the Fitzpatrick representative is
\begin{equation}
\mathcal{D}^{F_{\log}}(x,y) = \begin{cases} +\infty & \text{if~}  x\leq 0 \;\text{or}\;y\leq 0,\\
x\left(\W\left(\frac{xe}{y} \right)+\frac{1}{\W\left(\frac{xe}{y} \right)} -2 \right) & \text{otherwise}, \end{cases}
\end{equation}
where $\W$ is the principal branch of the Lambert $\W$ function. The above expression simplifies, except on the set $\{0\}\times\left[0,+\infty\right[$, to the lower closed version
\begin{align}\label{fitz:dist}
\overline{\mathcal{D}}^{F_{\log}}: \RR_+ \times \RR_+ & \rightarrow \left[0,+\infty\right]\\
(x,y) &\mapsto \begin{cases}
+\infty & \text{if~}  x<0 \;\text{or}\; y<0\; \text{or}\; (x>0 \; \text{and}\; y=0),\\
ye^{-1} & \text{if~}  x=0\; \text{and}\; y\geq 0,\\
x\left(\W\left(\frac{xe}{y} \right)+\frac{1}{\W\left(\frac{xe}{y} \right)} -2 \right) & \text{otherwise}.\nonumber
\end{cases}
\end{align}
This distance is shown in Figure~\eqref{D:fitz}. Taking the closure of the epigraph of $\mathcal{D}^{F_{\log}}$ admits $\overline{\mathcal{D}}^{F_{\log}}(0,y)=ye^{-1}$, the difference we see in \eqref{fitz:dist}. We opt to use the latter lower semicontinuous version of the distance in our later examples of proximity operators and envelopes. For information on the computation of $F_{\log}$ and $\mathcal{D}^{F_{\log}}$, see \cite[Example~.17.6]{BMcS06} and \cite[Example 3.2]{BDL19}, respectively.
\end{example}

\begin{example}[Kullback--Leibler divergence and GBD $\mathcal{D}^{\sigma_{\log}}$]
\label{ex:sigmalog}
The GBD constructed with the biggest member of $\mathcal{H}(\log)$ is
\begin{equation}
\mathcal{D}^{\sigma_{\log}}: (x,y) \mapsto -x\log(y) +\begin{cases}
x \log(x) & \text{if~} \log(y)\leq \log(x),\\
+\infty & \text{otherwise},
\end{cases}
\end{equation}
which may be recognized as equal, except at the point $(0,0)$, to the lower closed version
\begin{equation}\label{eqn:Dsigma}
\overline{\mathcal{D}}^{\sigma_{\log}}: (x,y) \mapsto  \begin{cases}
x \log(x) - x\log(y) & \text{if~} 0< y \leq x,\\
0 & \text{if~} x=y=0,\\
+\infty & \text{otherwise},
\end{cases}
\end{equation}
which is shown in Figure~\eqref{D:sigma}.
\end{example}

\section{Envelopes and proximity operators}
\label{sec:envelopes}

We next define formally the envelopes and the proximity operators.

\begin{definition}
\label{d:env&prox}
Given $\theta\colon X\to \left]-\infty, +\infty\right]$ and $\gamma\in \RR_{++}$, the \emph{left} and \emph{right $\mathcal{D}^{\star,h}_T$-envelopes} of $\theta$ with parameter $\gamma$ are respectively defined by
\begin{subequations}
\begin{align}
\label{e:leftenv}
\benv[\star,h]{\gamma,\theta}\colon X\to \left[-\infty, +\infty\right] \colon
y&\mapsto \inf_{x\in X} \left(\theta(x) +\frac{1}{\gamma}\mathcal{D}^{\star,h}_T(x,y)\right)\\
\label{e:rightenv} \text{and}\quad
\fenv[\star,h]{\gamma,\theta} \colon X\to \left[-\infty, +\infty\right] \colon
x&\mapsto
\inf_{y\in X} \left(\theta(y) +\frac{1}{\gamma}\mathcal{D}^{\star,h}_T(x,y)\right).
\end{align}
\end{subequations}
The \emph{left} and \emph{right proximity operator} of $\theta$ with parameter $\gamma$ are respectively defined by
\begin{subequations}
\begin{align}
\bprox{\gamma,\theta}^{\star,h} \colon X\rightrightarrows X \colon
y&\mapsto
\argmin_{x\in X} \left(\theta(x) +\frac{1}{\gamma}\mathcal{D}^{\star,h}_T(x,y)\right)\\
\text{and}\quad
\fprox{\gamma,\theta}^{\star,h} \colon X\rightrightarrows X \colon
x&\mapsto
\argmin_{y\in X} \left(\theta(y) +\frac{1}{\gamma}\mathcal{D}^{\star,h}_T(x,y)\right).
\end{align}
\end{subequations}
As before, the symbol $\star$ can be either $\flat$ or $\sharp$. When $T$ is point to point, we remove the symbol $\star$. When the symbol $\dagger$ appears next to the name of the representative function with the envelope or proximity operator, it is understood that the distance being used is the closed GBD $\overline{\mathcal{D}}_{T}^{\star,h}$. When the Bregman distance \eqref{d:Bregman_flat} or \eqref{d:Bregman_sharp} is used, we remove the name of the representative function in the envelopes and proximity operators.
\end{definition}

\begin{remark}[A selection operator for the proximity operator]
\label{r:env-prox}
We will find it useful to employ a selection map for the proximity operators.
\begin{enumerate}
\item\label{r:env-prox_left} 
If $\bproxs{\gamma,\theta}^{\star,h}(y)\in \bprox{\gamma,\theta}^{\star,h}(y)$, then
\begin{equation}
\benv[\star,h]{\gamma,\theta}(y) 
=\theta(\bproxs{\gamma,\theta}^{\star,h}(y)) +\frac{1}{\gamma}\mathcal{D}^{\star,h}_T(\bproxs{\gamma,\theta}^{\star,h}(y),y) 
\geq \theta(\bproxs{\gamma,\theta}^{\star,h}(y)).
\end{equation}

\item\label{r:env-prox_right} 
If $\fproxs{\gamma,\theta}^{\star,h}(x)\in \fprox{\gamma,\theta}^{\star,h}(x)$, then
\begin{equation}
\fenv[\star,h]{\gamma,\theta}(x) 
=\theta(\fproxs{\gamma,\theta}^{\star,h}(x)) +\frac{1}{\gamma}\mathcal{D}^{\star,h}_T(x,\fproxs{\gamma,\theta}^{\star,h}(x)) 
\geq \theta(\fproxs{\gamma,\theta}^{\star,h}(x)).
\end{equation}
\end{enumerate}
Sufficient conditions for the images of the proximity operators to be nonempty will be provided in Propositions~\ref{p:left prox} and \ref{p:right prox}; we only use selection operators in such cases.
\end{remark}

\begin{example}[Energy]
\label{ex:env:energy}
Let $f$ be the energy and $\theta : x \mapsto |x-1/2|$. In this case, $\partial f = \Id$, and the distances $\mathcal{D}^{F_{\Id}}$ and $\mathcal{D}^{\sigma_{\Id}}$ are given in Example~\ref{ex:dist:energy} and are both lsc. It is straightforward to determine that their corresponding envelopes and proximity operators are given by:
\begin{enumerate}
\item $\mathcal{D}^{F_{\Id}}$: the corresponding \emph{GBD envelopes} (both left $\benv[F_{\Id}]{\gamma,\theta}$ and right $\fenv[F_{\Id}]{\gamma,\theta}$) are equal to the \emph{Moreau envelope} with parameter $2\gamma$, and the corresponding proximity operators are equal to the respective Moreau proximity operators with parameter $2\gamma$.
\item $\mathcal{D}^{\sigma_{\Id}}$: the left and right proximity operators are both necessarily $\Id$, and our left and right envelopes for $\theta$ are exactly equal to $\theta$ for both the right and left envelopes, a fact that would hold true with any other choice of $\theta$.
\end{enumerate}
\end{example}

In the following result, part \ref{p:domain_scale} extends \cite[Proposition~2.1]{BDL17} and its proof is the same as the one in \cite[Proposition~12.22(i)]{BC17}. Parts \ref{p:domain_dom} and \ref{p:domain_dom'} are new and establish relationships between the domain of left and right envelopes with those of $S$ and $T$.

\begin{proposition}
\label{p:domain}
Let $\theta\colon X\to \left]-\infty, +\infty\right]$ and let $\gamma, \mu\in \RR_{++}$. Then the following hold:
\begin{enumerate}
\item\label{p:domain_scale} 
$\benv[\star,h]{\mu,\gamma\theta} =\gamma\benv[\star,h]{\gamma\mu,\theta}$ and $\fenv[\star,h]{\mu,\gamma\theta} =\gamma\fenv[\star,h]{\gamma\mu,\theta}$.
\item\label{p:domain_dom}
$\dom\benv[\star,h]{\gamma,\theta} \subseteq \dom T$ and $\dom\fenv[\star,h]{\gamma,\theta} \subseteq \dom S$. Moreover,
\begin{subequations}
\begin{align}
\dom\benv[\sharp,h]{\gamma,\theta} &=\bigcup_{x\in \dom\theta,\ \varepsilon\in \RP} T^{-1}(S^h_\varepsilon x) \quad\text{and}\\
\dom\fenv[\sharp,h]{\gamma,\theta} &=\left\{x\in X: \dom\theta\cap \bigcup_{\varepsilon\in \RP} T^{-1}(S^h_\varepsilon x)\neq \varnothing\right\}.
\end{align}
\end{subequations}
\item\label{p:domain_dom'}
If $\dom S\cap \dom\theta\neq \varnothing$ and $(\dom S\cap \dom\theta)\times \dom T\subseteq \dom \mathcal{D}^{\star,h}_T$, then $\dom\benv[\star,h]{\gamma,\theta} =\dom T$. If $\dom T\cap \dom\theta\neq \varnothing$ and $\dom S\times (\dom T\cap \dom\theta)\subseteq \dom \mathcal{D}^{\star,h}_T$, then $\dom\fenv[\star,h]{\gamma,\theta} =\dom S$.
\end{enumerate}
\end{proposition}
\begin{proof}
\ref{p:domain_scale}: This is straightforward from the definition.

\ref{p:domain_dom}: We observe that 
\begin{subequations}
\label{e:dom}
\begin{align}
y\in \dom\benv[\star,h]{\gamma,\theta} 
&\iff \inf_{x\in X} \left(\theta(x) +\frac{1}{\gamma}\mathcal{D}^{\star,h}_T(x,y)\right) =\benv[\star,h]{\gamma,\theta}(y) <+\infty\\
&\iff \exists x\in X,\quad \theta(x) +\frac{1}{\gamma}\mathcal{D}^{\star,h}_T(x,y) <+\infty\\
&\label{e:domD} \iff \exists x\in \dom\theta,\quad \mathcal{D}^{\star,h}_T(x,y) <+\infty, 
\end{align}
\end{subequations}
since it always holds that $\theta(x) >-\infty$ and $\mathcal{D}^{\star,h}_T(x,y)\geq 0$.
On the one hand, \eqref{e:domD} implies that $y\in \dom T$, and hence $\dom\benv[\star,h]{\gamma,\theta} \subseteq \dom T$. To prove the first equality in (ii), note that, by definition of $\mathcal{D}^{\sharp,h}_T$,
\begin{subequations}
\begin{align}
\mathcal{D}^{\sharp,h}_T(x,y) <+\infty 
&\iff \sup_{v\in Ty} \left(h(x,v) -\scal{x}{v}\right) <+\infty\\
&\iff \exists\varepsilon\in \RP,\ \forall v\in Ty,\quad h(x,v) -\scal{x}{v}\leq \varepsilon\\
&\iff \exists\varepsilon\in \RP,\quad Ty\subseteq S^h_\varepsilon x\\
&\iff y\in \bigcup_{\varepsilon\in \RP} T^{-1}(S^h_\varepsilon x).
\end{align}
\end{subequations}
Combining with \eqref{e:dom} yields
\begin{equation}
\dom\benv[\sharp,h]{\gamma,\theta} =\bigcup_{x\in \dom\theta,\ \varepsilon\in \RP} T^{-1}(S^h_\varepsilon x).
\end{equation}
The proof for the right envelope is analogous. 

\ref{p:domain_dom'}: We will only prove the first claim because the second one is similar. In view of \ref{p:domain_dom}, it suffices to prove that $\dom T\subseteq \dom\benv[\star,h]{\gamma,\theta}$. Let $y\in \dom T$. By assumption, there exists $x_0\in \dom S\cap \dom\theta$. Then $(x_0,y)\in (\dom S\cap \dom\theta)\times \dom T\subseteq \dom \mathcal{D}^{\star,h}_T$ which gives $\mathcal{D}^{\star,h}_T(x_0,y) <+\infty$. Using \eqref{e:dom}, we obtain that $y\in \dom\benv[\star,h]{\gamma,\theta}$. This completes the proof.  
\end{proof}

The following proposition compares minimum values of the envelopes with those of the reference function $\theta$ over the domains of $S$ and $T$.

\begin{proposition}
\label{p:inequalities}
Let $\theta\colon X\to \left]-\infty, +\infty\right]$, $x\in \dom S$, $y\in \dom T$, and $\gamma, \mu\in \RR_{++}$ with $\gamma\leq \mu$. Then the following hold:
\begin{enumerate}
\item\label{p:inequalities_left}
Suppose that $\dom S\cap \dom\theta \neq\varnothing$ and that $Tz\subseteq Sz$ for all $z\in \dom T$. Then
\begin{subequations}
\begin{align}
\label{p:inequalities eq1}
&\inf\theta(\dom S)\leq \benv[\star,h]{\mu,\theta}(y)\leq \benv[\star,h]{\gamma,\theta}(y)\leq \theta(y)\\
\label{p:inequalities eq1'}
\text{and}\quad &\inf\theta(\overline{\dom}\, S)\leq \benv[\star,\dagger h]{\mu,\theta}(y)\leq \benv[\star,\dagger h]{\gamma,\theta}(y)\leq \benv[\star,h]{\gamma,\theta}(y)\leq \theta(y).
\end{align}
\end{subequations} 
Consequently, $\inf\theta(\dom S)\leq \inf\benv[\star,h]{\gamma,\theta}(X)\leq \inf\theta(\dom T)$, with equality throughout when $\dom S\subseteq \dom T$. Moreover, there exist $\alpha, \beta\in \left[\inf\theta(\dom S), \theta(y)\right]$ such that
\begin{equation}
\benv[\star,h]{\gamma,\theta}(y)\downarrow \alpha \text{~~as~~} \gamma\uparrow +\infty 
\quad\text{and}\quad
\benv[\star,h]{\gamma,\theta}(y)\uparrow \beta \text{~~as~~} \gamma\downarrow 0.   
\end{equation}

\item\label{p:inequalities_right}
Suppose that $\dom T\cap \dom\theta \neq\varnothing$ and that $Tz\subseteq Sz$ for all $z\in \dom S$. Then
\begin{subequations}
\begin{align}
\label{p:inequalities eq2}
&\inf\theta(\dom T)\leq \fenv[\star,h]{\mu,\theta}(x)\leq \fenv[\star,h]{\gamma,\theta}(x)\leq \theta(x)\\ 
\text{and}\quad &\inf\theta(\overline{\dom}\, T)\leq \fenv[\star,\dagger h]{\mu,\theta}(x)\leq \fenv[\star,\dagger h]{\gamma,\theta}(x)\leq \fenv[\star,h]{\gamma,\theta}(x)\leq \theta(x).
\end{align}
\end{subequations} 
Consequently, $\inf\theta(\dom T)\leq \inf\fenv[\star,h]{\gamma,\theta}(X)\leq \inf\theta(\dom S)$, with equality throughout when $\dom T\subseteq \dom S$. Moreover, there exist $\alpha, \beta\in \left[\inf\theta(\dom T), \theta(x)\right]$ such that
\begin{equation}
\fenv[\star,h]{\gamma,\theta}(x)\downarrow \alpha \text{~~as~~} \gamma\uparrow +\infty 
\quad\text{and}\quad
\fenv[\star,h]{\gamma,\theta}(x)\uparrow \beta \text{~~as~~} \gamma\downarrow 0.   
\end{equation}
\end{enumerate}
\end{proposition}
\begin{proof}
The proof for the right envelopes follows the same steps as those for the left ones, and hence, the details are omitted. The assumption of $Tz\subseteq Sz$ for all $z\in \dom T$, together with \cite[Remark~3.3(c)]{BM-L18}, implies that
\begin{equation}
\label{e2:x=y}
\forall z\in \dom T,\quad \mathcal{D}^{\star,h}_T(z,z) =0.
\end{equation}
As $\mu \geq \gamma >0$, we have that, for all $x\in \dom S$,
\begin{equation}
\theta(x)\leq \theta(x) +\frac{1}{\mu}\mathcal{D}^{\star,h}_T(x,y)\leq \theta(x) +\frac{1}{\gamma}\mathcal{D}^{\star,h}_T(x,y).    
\end{equation}
Taking the infimum over $x\in \dom S$, with noting that $\mathcal{D}^{\star,h}_T(x,y) =+\infty$ for $x\notin \dom S$, yields
\begin{subequations}
\begin{align}
\inf\theta(\dom S) &\leq \inf_{x\in \dom S} \theta(x) +\frac{1}{\mu}\mathcal{D}^{\star,h}_T(x,y)= \inf_{x\in X} \theta(x) +\frac{1}{\mu}\mathcal{D}^{\star,h}_T(x,y)= \benv[\star,h]{\mu,\theta}(y)\\
&\leq \inf_{x\in \dom S} \theta(x) +\frac{1}{\gamma}\mathcal{D}^{\star,h}_T(x,y)=\inf_{x\in X} \theta(x) +\frac{1}{\gamma}\mathcal{D}^{\star,h}_T(x,y)= \benv[\star,h]{\gamma,\theta}(y)\\
&\leq \theta(y) +\frac{1}{\gamma}\mathcal{D}^{\star,h}_T(y,y)=\theta(y),    
\end{align}
\end{subequations}
where the last equality is due to \eqref{e2:x=y}. This proves \eqref{p:inequalities eq1}. 

Next, for all $x\in \overline{\dom}\, S$, it holds that $\theta(x)\leq \theta(x) +\frac{1}{\mu}\overline{\mathcal{D}}^{\star,h}_T(x,y)\leq \theta(x) +\frac{1}{\gamma}\overline{\mathcal{D}}^{\star,h}_T(x,y)$, and hence 
\begin{equation}\label{e:closed}
\inf\theta(\overline{\dom}\, S) =\inf_{x\in \overline{\dom}\, S} \theta(x)\leq \benv[\star,\dagger h]{\mu,\theta}(y)\leq \benv[\star,\dagger h]{\gamma,\theta}(y),
\end{equation}
where we used the fact that $\overline{\mathcal{D}}^{\star,h}_T(x,y) =+\infty$ for $x\notin \overline{\dom}\, S$. Noting also that $\overline{\mathcal{D}}^{\star,h}_T \leq \mathcal{D}^{\star,h}_T$, we have  $\benv[\star,\dagger h]{\gamma,\theta}(y) \leq \benv[\star,h]{\gamma,\theta}(y)$, which combined with \eqref{e:closed} and \eqref{p:inequalities eq1} implies \eqref{p:inequalities eq1'}.   

Now, taking infimum over $y\in \dom T$ in \eqref{p:inequalities eq1} and using the fact that $\dom\benv[\star,h]{\gamma,\theta}\subseteq \dom T$ (see Proposition~\ref{p:domain}\ref{p:domain_dom}), we obtain that 
\begin{equation}\label{e:infinf}
\inf\theta(\dom S)\leq \inf\benv[\star,h]{\gamma,\theta}(X)\leq \inf\theta(\dom T).   
\end{equation}
If $\dom S\subseteq \dom T$, then $\inf\theta(\dom S)\geq \inf\theta(\dom T)$, and the equalities in \eqref{e:infinf} must hold. The remaining conclusion follows directly from \eqref{p:inequalities eq1} and the fact that $\gamma\leq \mu$.
\end{proof}

\begin{remark}\label{r:optim}
Suppose that $\dom T =\dom S =:D$ and we want to solve the optimization problem
\begin{equation}
\min \theta(x) \quad\text{s.t.}\quad x\in D.    
\end{equation}
It is interesting to be able to state relationships between the optimal value of the problem and $\inf \benv[\star,h]{\gamma,\theta}(D)$ as well as $\inf \fenv[\star,h]{\gamma,\theta}(D)$. It is also important to establish relationships between the sets $\argmin \theta(D)$, $\argmin \benv[\star,h]{\gamma,\theta}(D)$, and $\argmin \fenv[\star,h]{\gamma,\theta}(D)$. We establish these relationships in the next result. 
\end{remark}

\begin{proposition}
\label{p:basic}
Suppose that $\dom T =\dom S =:D$ and that $Tz\subseteq Sz$ for all $z\in D$. 
Let $\theta\colon X\to \left]-\infty, +\infty\right]$ with $D\cap \dom\theta \neq\varnothing$, let $\gamma\in \RR_{++}$, and let $x, y\in D$. Set 
\begin{subequations}
\begin{align}
&\overleftarrow{A} :=\bigcup_{x\in \argmin\theta(D)} \{y\in D: Ty\cap Sx\neq \varnothing\},\quad 
\overleftarrow{B} :=\bigcup_{x\in \argmin\theta(D)} \{y\in D: Ty\subseteq Sx\}\\
\text{and~~}&\overrightarrow{A} :=\bigcup_{y\in \argmin\theta(D)} \{x\in D: Ty\cap Sx\neq \varnothing\},\quad 
\overrightarrow{B} :=\bigcup_{y\in \argmin\theta(D)} \{x\in D: Ty\subseteq Sx\}.
\end{align}    
\end{subequations}
Then the following hold:
\begin{enumerate}
\item\label{p:basic_min}
$\inf\theta(D) =\inf\benv[\star,h]{\gamma,\theta}(X) =\inf\fenv[\star,h]{\gamma,\theta}(X)$ and $\argmin\theta(D)\subseteq \argmin\benv[\star,h]{\gamma,\theta}(X)\cap \argmin\fenv[\star,h]{\gamma,\theta}(X)$.

\item\label{p:basic_flat}
$\overleftarrow{A}\subseteq \argmin\benv[\flat,h]{\gamma,\theta}(X)$ and $\overrightarrow{A}\subseteq \argmin\fenv[\flat,h]{\gamma,\theta}(X)$.

\item\label{p:basic_sharp}
If $\bprox{\gamma,\theta}^{\sharp,h}(y)\neq \varnothing$ for all $y\in \argmin \benv[\sharp,h]{\gamma,\theta}(X)$, then $\argmin\benv[\sharp,h]{\gamma,\theta}(X) \subseteq \overleftarrow{B}$. Consequently, $\argmin\benv[\sharp,h]{\gamma,\theta}(X) \subseteq \overleftarrow{B}\subseteq \overleftarrow{A} \subseteq \argmin\benv[\flat,h]{\gamma,\theta}(X)$.

\item\label{p:basic_sharp_right}
If $\fprox{\gamma,\theta}^{\sharp,h}(y)\neq \varnothing$ for all $y\in \argmin \fenv[\sharp,h]{\gamma,\theta}(X)$, then $\argmin\fenv[\sharp,h]{\gamma,\theta}(X) \subseteq \overrightarrow{B}$. Consequently, $\argmin\fenv[\sharp,h]{\gamma,\theta}(X) \subseteq \overrightarrow{B}\subseteq \overrightarrow{A} \subseteq \argmin\fenv[\flat,h]{\gamma,\theta}(X)$.
\end{enumerate}
\end{proposition}
\begin{proof}
\ref{p:basic_min}: Since $\dom T =\dom S =D$, Proposition~\ref{p:inequalities} implies that $\inf\theta(D) =\inf\benv[\star,h]{\gamma,\theta}(X) =\inf\fenv[\star,h]{\gamma,\theta}(X)$. Now, let $z\in \argmin\theta(D)$. Then $z\in D$ and, by assumption, $Tz\subseteq Sz$, from which we have $\mathcal{D}^{\star,h}_T(z,z) =0$. Therefore, $\theta(z)+\frac{1}{\gamma}\mathcal{D}_T^{\star,h}(z,z) =\theta(z) =\inf\theta(D) =\inf\benv[\star,h]{\gamma,\theta}(X) =\inf\fenv[\star,h]{\gamma,\theta}(X)$, which implies that $z\in \argmin\benv[\star,h]{\gamma,\theta}(X)$ and also $z\in \argmin\fenv[\star,h]{\gamma,\theta}(X)$. We obtain that $\argmin\theta(D)\subseteq \argmin\benv[\star,h]{\gamma,\theta}(X)\cap \argmin\fenv[\star,h]{\gamma,\theta}(X)$. 

\ref{p:basic_flat}: Let any $y\in \overleftarrow{A}$. Then $y\in D$ and there exists $x\in \argmin\theta(D)$ such that $Ty\cap Sx\neq \varnothing$. By \cite[Remark~3.3(b)]{BM-L18}, $\mathcal{D}_T^{\flat,h}(x,y) =0$. Using \ref{p:basic_min} with $\star=\flat$, we have
\begin{equation}
\inf\theta(D) =\inf\benv[\flat,h]{\gamma,\theta}(X) \leq \benv[\flat,h]{\gamma,\theta}(y)\leq \theta(x) +\frac{1}{\gamma}\mathcal{D}^{\flat,h}_T(x,y) =\theta(x) =\inf\theta(D),  
\end{equation}
which implies that $y\in \argmin\benv[\flat,h]{\gamma,\theta}(X)$. Hence, $\overleftarrow{A}\subseteq \argmin\benv[\flat,h]{\gamma,\theta}(X)$. Similarly, we have that $\overrightarrow{A}\subseteq \argmin\fenv[\flat,h]{\gamma,\theta}(X)$. 

\ref{p:basic_sharp}: Let any $y\in \argmin\benv[\sharp,h]{\gamma,\theta}(X)$. By assumption, there exists $x\in X$ such that $\theta(x) +\frac{1}{\gamma}\mathcal{D}^{\sharp,h}_T(x,y) =\benv[\sharp,h]{\gamma,\theta}(y) <+\infty$. Then, we must have $x\in \dom S =D$ and, by \ref{p:basic_min} with $\star=\sharp$,
\begin{equation}
\inf\theta(D) =\inf\benv[\sharp,h]{\gamma,\theta}(X) =\benv[\sharp,h]{\gamma,\theta}(y) =\theta(x) +\frac{1}{\gamma}\mathcal{D}^{\sharp,h}_T(x,y) \geq \theta(x)\geq \inf\theta(D).    
\end{equation}
Therefore, $\theta(x) =\inf \theta(D)$ and $\mathcal{D}^{\sharp,h}_T(x,y) =0$. While the former implies $x\in \argmin\theta(D)$, the latter implies $Ty\subseteq Sx$ (see \cite[Proposition~3.7(b)]{BM-L18}). We deduce that $y\in \overleftarrow{B}$.

\ref{p:basic_sharp_right}: This is similar to \ref{p:basic_sharp}. 
\end{proof}

Proposition~\ref{p:basic} shows that the envelopes can be used to automatically impose the constraints. Namely, they transform the original constrained problem into an unconstrained one. This remark justifies the assumption imposed in the next proposition.

For $A\subseteq X$, denote by $\overline{A}^w$ its weak closure. Namely, $\overline{A}^w$ contains all the (weak) limits of weakly convergent sequences contained in $A$. Recall that a set is said to be \emph{precompact} when its closure is compact. 

\begin{proposition}[Left proximity operators]
\label{p:left prox}
Let $\theta\in \Gamma_0(X)$ and $y\in \dom T$. Suppose that $\dom S\cap \dom\theta\neq \varnothing$ and $(\dom S\cap \dom\theta)\times \{y\}\subseteq \dom \mathcal{D}^{\sharp,h}_T$. Then the following hold:
\begin{enumerate}
\item\label{p:left prox_theta+}
Suppose that $\varphi_\mu(\cdot) :=\theta(\cdot) +\frac{1}{\mu}\mathcal{D}^{\sharp,h}_T(\cdot,y)$ is coercive for some $\mu\in \RR_{++}$. Then, for all $\gamma\in \left]0,\mu \right]$, $\varnothing\neq \bprox{\gamma,\theta}^{\sharp,h}(y)\subseteq \dom S\cap \dom\theta$. If $Tz\subseteq Sz$ for all $z\in \dom T$, then 
\begin{equation}
P :=\bigcup_{\gamma\in \left]0,\mu \right]} \bprox{\gamma,\theta}^{\sharp,h}(y)\subseteq \lev{\theta(y)} \varphi_\mu.
\end{equation}
If, in addition, $y \in \dom \theta$, then $P$ is weakly precompact.

\item\label{p:left prox_theta}
Suppose that $\theta$ is coercive. Then, for all $\gamma\in \RR_{++}$, $\varnothing\neq \bprox{\gamma,\theta}^{\sharp,h}(y)\subseteq \dom S\cap \dom\theta$. If $Tz\subseteq Sz$ for all $z\in \dom T$, then 
\begin{equation}
P :=\bigcup_{\gamma\in \RR_{++}} \bprox{\gamma,\theta}^{\sharp,h}(y)\subseteq \lev{\theta(y)} \theta.
\end{equation}
If, in addition, $y \in \dom \theta$, then $P$ is weakly precompact.
\end{enumerate}
\end{proposition}
\begin{proof} 
For each $\gamma\in \RR_{++}$, set $\varphi_\gamma(\cdot) :=\theta(\cdot) +\frac{1}{\gamma}D^{\sharp,h}_T(\cdot,y)$. It follows from $\dom S\cap \dom\theta\neq \varnothing$ and $(\dom S\cap \dom\theta)\times \{y\}\subseteq \dom \mathcal{D}^{\sharp,h}_T$ that $\varphi_\gamma$ is proper. By the same argument as in \cite[Lemma~3.17(b)]{BM-L18}, $\mathcal{D}^{\sharp,h}_T(\cdot,y)$ is (strongly) lsc on $X$, and so is $\varphi_\gamma$. Since convexity of $\varphi_\gamma$ follows from that of $\theta$ and $\mathcal{D}^{\sharp,h}_T(\cdot,y)$, we obtain that $\varphi_\gamma\in \Gamma_0(X)$.

\ref{p:left prox_theta+}: Let $\gamma\in \left]0,\mu \right]$. Then $\varphi_\gamma\geq \varphi_{\mu}$, and the coercivity of $\varphi_\gamma$ follows from the assumption that $\varphi_\mu$ is coercive. By combining with the fact that $\varphi_\gamma\in \Gamma_0(X)$ and using \cite[Theorem~5.4.4]{Schirotzek2007}, there exists $z_\gamma\in X$ such that $\varphi_\gamma(z_\gamma) =\min_{x\in X} \varphi_\gamma(x)< +\infty$. We obtain that $\theta(z_\gamma) +\frac{1}{\gamma}\mathcal{D}^{\sharp,h}_T(z_\gamma,y) =\benv[\sharp,h]{\gamma,\theta}(y) <+\infty$, and so $\bprox{\gamma,\theta}^{\sharp,h}(y)\neq \varnothing$. Now, take an arbitrary $s_\gamma\in \bprox{\gamma,\theta}^{\sharp,h}(y)$. Then $\theta(s_\gamma) +\frac{1}{\gamma}\mathcal{D}^{\sharp,h}_T(s_\gamma,y) =\benv[\sharp,h]{\gamma,\theta}(y) <+\infty$. Since $\theta(s_\gamma) >-\infty$ and $\mathcal{D}^{\sharp,h}_T(s_\gamma,y)\geq 0$, we must have that $\mathcal{D}^{\sharp,h}_T(s_\gamma,y) <+\infty$ and $\theta(s_\gamma) <+\infty$, which yield $s_\gamma\in \dom S\cap \dom\theta$. Hence, $\bprox{\gamma,\theta}^{\sharp,h}(y)\subseteq \dom S\cap \dom\theta$. 

Next, assume that $Tz\subseteq Sz$ for all $z\in \dom T$. We have from the rightmost inequality in \eqref{p:inequalities eq1} of Proposition~\ref{p:inequalities}\ref{p:inequalities_left} that, for all $\gamma\in \left]0,\mu \right]$ and $s_\gamma\in \bprox{\gamma,\theta}^{\sharp,h}(y)$,  
\begin{equation}
\theta(y)\geq \benv[\sharp,h]{\gamma,\theta}(y) =\theta(s_{\gamma}) +\frac{1}{\gamma}\mathcal{D}^{\sharp,h}_T(s_{\gamma},y)\geq \varphi_\mu(s_\gamma) =\theta(s_{\gamma}) +\frac{1}{\mu}\mathcal{D}^{\sharp,h}_T(s_{\gamma},y).
\end{equation}
Therefore, 
\begin{equation}
P =\bigcup_{\gamma\in \left]0,\mu \right]} \bprox{\gamma,\theta}^{\sharp,h}(y)\subseteq  \lev{\theta(y)} \varphi_\mu =\{z\in X: \varphi_\mu(z)\leq \theta(y)\}.
\end{equation}
Now, assume that $y \in \dom \theta$, so $\theta(y) < +\infty$. Since $\varphi_\mu$ is convex and lsc, it is weakly lsc. Moreover, $\varphi_\mu$ is coercive, so its level set $\lev{\theta(y)} \varphi_\mu$ is bounded and weakly closed. By \cite[Theorem 3.17]{Brezis2011}, $\lev{\theta(y)} \varphi_\mu$ is weakly compact. This directly implies that $\overline{P}^w \subseteq \lev{\theta(y)} \varphi_\mu$. Being a weakly closed subset of a weakly compact set, $\overline{P}^w$ is also weakly compact. Hence, $P$ is weakly precompact.

\ref{p:left prox_theta}: Since $\theta$ is coercive, we have that $\varphi_\mu$ is also coercive for all $\mu\in \RR_{++}$. The first conclusion follows from \ref{p:left prox_theta+}. Now, assume that $Tz\subseteq Sz$ for all $z\in \dom T$. Then, by the rightmost inequality in \eqref{p:inequalities eq1} of Proposition~\ref{p:inequalities}\ref{p:inequalities_left}, for all $\gamma\in \RR_{++}$ and $s_\gamma\in \bprox{\gamma,\theta}^{\sharp,h}(y)$,    
\begin{equation}
\theta(y)\geq \benv[\sharp,h]{\gamma,\theta}(y) =\theta(s_{\gamma}) +\frac{1}{\gamma}\mathcal{D}^{\sharp,h}_T(s_{\gamma},y)\geq \theta(s_{\gamma}).
\end{equation}
The rest of the proof is similar to that of \ref{p:left prox_theta+}.
\end{proof}

\begin{proposition}[Right proximity operators]
\label{p:right prox}
Let $\theta\colon X\to \left]-\infty, +\infty\right]$ be proper and lsc, and let $x\in \dom S$. Suppose that $X$ is finite-dimensional, that $\dom T\cap \dom\theta\neq \varnothing$, that $\{x\}\times (\dom T\cap \dom\theta)\subseteq \dom \mathcal{D}^{\star,h}_T$, and that $\mathcal{D}^{\star,h}_T(x,\cdot)$ is lsc. Then the following hold:
\begin{enumerate}
\item\label{p:right prox_theta+}
Suppose that $\varphi_\mu(\cdot) :=\theta(\cdot) +\frac{1}{\mu}\mathcal{D}^{\star,h}_T(x,\cdot)$ is coercive for some $\mu\in \RR_{++}$. Then, for all $\gamma\in \left]0,\mu \right]$, $\varnothing\neq \fprox{\gamma,\theta}^{\star,h}(x)\subseteq \dom T\cap \dom\theta$. If $Tz\subseteq Sz$ for all $z\in \dom S$, then 
\begin{equation}
P :=\bigcup_{\gamma\in \left]0,\mu \right]} \fprox{\gamma,\theta}^{\star,h}(x)\subseteq \lev{\theta(x)} \varphi_\mu.
\end{equation}
If, in addition, $x \in \dom \theta$, then $P$ is bounded.

\item\label{p:right prox_theta}
Suppose that $\theta$ is coercive. Then, for all $\gamma\in \RR_{++}$, $\varnothing\neq \fprox{\gamma,\theta}^{\star,h}(x)\subseteq \dom T\cap \dom\theta$. If $Tz\subseteq Sz$ for all $z\in \dom S$, then 
\begin{equation}
P :=\bigcup_{\gamma\in \RR_{++}} \fprox{\gamma,\theta}^{\star,h}(x)\subseteq \lev{\theta(x)} \theta.
\end{equation}
If, in addition, $x \in \dom \theta$, then $P$ is bounded. 
\end{enumerate}
\end{proposition}
\begin{proof}
Arguing as in the proof of Proposition~\ref{p:left prox}, for all $\gamma\in \left]0,\mu \right]$ in case \ref{p:right prox_theta+} and for all $\gamma\in \RR_{++}$ in case \ref{p:right prox_theta}, $\varphi_\gamma(\cdot) :=\theta(\cdot) +\frac{1}{\gamma}D^{\star,h}_T(x,\cdot)$ is proper, lsc, and coercive. Since $\fenv[\star,h]{\gamma,\theta}(x) =\inf_{y\in X} \varphi_\gamma(y) <+\infty$, we can take a sequence $(y_n)_{n\in \mathbb{N}}$ in $X$ such that $\varphi_\gamma(y_n)\to \fenv[\star,h]{\gamma,\theta}(x)$ as $n\to +\infty$. We derive from the coercivity of $\varphi_\gamma$ that $(y_n)_{n\in \mathbb{N}}$ is bounded, and so there is a subsequence $(y_{k_n})_{n\in \mathbb{N}}$ converging to some $z_\gamma\in X$. As $\varphi_\gamma$ is lsc, $\varphi_\gamma(z_\gamma)\leq \liminf_{n\to +\infty} \varphi_\gamma(y_n) =\fenv[\star,h]{\gamma,\theta}(x) =\inf_{y\in X} \varphi_\gamma(y)\leq \varphi_\gamma(z_\gamma)$. Therefore, $\fenv[\star,h]{\gamma,\theta}(x) =\varphi_\gamma(z_\gamma) =\theta(z_\gamma) +\frac{1}{\gamma}\mathcal{D}^{\star,h}_T(x,z_\gamma)$, which implies that $\fprox{\gamma,\theta}^{\star,h}(x)\neq \varnothing$. Proceeding as in the proof of Proposition~\ref{p:left prox}\ref{p:left prox_theta+}, we have $\fprox{\gamma,\theta}^{\star,h}(x)\subseteq \dom T\cap \dom\theta$. 

We now prove the second conclusion of \ref{p:right prox_theta+}. By the assumption that $Tz\subseteq Sz$ for all $z\in \dom S$, the rightmost inequality in \eqref{p:inequalities eq2} of Proposition~\ref{p:inequalities}\ref{p:inequalities_right} implies that, for all $\gamma\in \left]0,\mu \right]$ and $s_\gamma\in \fprox{\gamma,\theta}^{\star,h}(x)$,  
\begin{equation}
\theta(x)\geq \fenv[\star,h]{\gamma,\theta}(x) =\theta(s_{\gamma}) +\frac{1}{\gamma}\mathcal{D}^{\star,h}_T(x,s_{\gamma})\geq \varphi_\mu(s_\gamma) =\theta(s_{\gamma}) +\frac{1}{\mu}\mathcal{D}^{\star,h}_T(x,s_{\gamma}).
\end{equation}
We deduce that 
\begin{equation}
P =\bigcup_{\gamma\in \left]0,\mu \right]} \fprox{\gamma,\theta}^{\star,h}(x)\subseteq  \lev{\theta(x)} \varphi_\mu =\{z\in X: \varphi_\mu(z)\leq \theta(x)\}. 
\end{equation}
Combining this fact with the coercivity of $\varphi_\mu$ and the additional assumption that $\theta(x) < +\infty$, we obtain the boundedness of $P$. The second conclusion of \ref{p:right prox_theta} follows by the same argument as in the proof of Proposition~\ref{p:left prox}\ref{p:left prox_theta}.
\end{proof}

\section{Asymptotic behaviour properties}\label{sec:asymptotic}

Moreau first considered the envelope that has come to bear his name in the setting of $\gamma=1$ \cite{Mor62}. He was interested, in particular, in the characterization of infimal convolution as epigraph addition; see \cite{RW98}. Attouch introduced the more general parameter $\gamma$ for regularizing convex functions \cite{Attouch77,Attouch84} and later with Wets for nonconvex functions \cite{AttouchWets}. When the regularized function is the sum of a convex objective function $\theta$ together with the indicator function for a constraint set, the Moreau envelope provides a smooth regularization with full domain. Recovery of the regularized function $\theta$ as $\gamma\downarrow 0$ is important for algorithms that use the Moreau envelope as a surrogate for $\theta$, while the asymptotic properties as $\gamma\uparrow +\infty$ shed light on other properties of the regularization. In what follows, we will analyse both. 

We say that a maximally monotone operator $S$ is \emph{strictly monotone over a subset $B\subseteq \dom S$}  if for every $z,z'\in B$ we have
\begin{equation}
\langle z- z', w-w'\rangle=0 \text{~with~} w\in Tz,\, w'\in Tz' \text{~implies~} z=z'.     
\end{equation}

Because the proximity operators are generically set-valued operators, when their images are nonempty we will make use of selection operators $\bproxs{\gamma,\theta}^{\star,h},\fproxs{\gamma,\theta}^{\star,h}$ that satisfy	$\bproxs{\gamma,\theta}^{\star,h}(y)\in \bprox{\gamma,\theta}^{\star,h}(y)$ and $\fproxs{\gamma,\theta}^{\star,h}(x)\in \fprox{\gamma,\theta}^{\star,h}(x)$.

\begin{theorem}[Asymptotic left behaviour when $\gamma\downarrow 0$] \label{t:to0}
Let $\theta\in \Gamma_0(X)$ and let $y\in \dom T\cap \dom\theta$. Suppose that $\dom S\cap \dom\theta\neq \varnothing$, that $(\dom S\cap \dom\theta)\times \{y\}\subseteq \dom \mathcal{D}^{\sharp,h}_T$, that $Tz\subseteq Sz$ for all $z\in \dom T$, and that $\varphi_\mu(\cdot) :=\theta(\cdot) +\frac{1}{\mu}\mathcal{D}^{\sharp,h}_T(\cdot,y)$ is coercive for some $\mu\in \RR_{++}$. For each $\gamma\in \left]0,\mu \right]$, let $s_\gamma :=\bproxs{\gamma,\theta}^{\sharp,h}(y)\in \bprox{\gamma,\theta}^{\sharp,h}(y)$. Then the following hold:
\begin{enumerate}
\item\label{t:to0_limD}
$\mathcal{D}^{\sharp,h}_T(s_{\gamma},y)\to 0$ as $\gamma\downarrow 0$.

\item\label{t:to0_D=0} 
For every weak cluster point $z$ of $(s_\gamma)_{\gamma\in \left]0,\mu \right]}$ as $\gamma\downarrow 0$, it holds that $\mathcal{D}^{\sharp,h}_T(z,y)=0$, $Ty\subseteq Sz$, and $z\in \lev{\theta(y)}\theta$. 

\item\label{t:to0_lim} 
If $T=S$ and $S$ is strictly monotone over $\dom S\cap \dom \theta$, then, as $\gamma\downarrow 0$,
\begin{equation}
s_\gamma\rightharpoonup y,\quad \benv[\sharp,h]{\gamma,\theta}(y)\uparrow \theta(y),\quad \theta(s_{\gamma})\to \theta(y), \quad\text{and}\quad \frac{1}{\gamma}\mathcal{D}^{\sharp,h}_T(s_{\gamma},y)\to 0.    
\end{equation}
\end{enumerate}
\end{theorem}
\begin{proof}
By the rightmost inequality in \eqref{p:inequalities eq1} of Proposition~\ref{p:inequalities}\ref{p:inequalities_left}, for all $\gamma\in \left]0,\mu \right]$,  
\begin{equation}\label{e:gamma<mu}
\theta(y)\geq \benv[\sharp,h]{\gamma,\theta}(y) =\theta(s_{\gamma}) +\frac{1}{\gamma}\mathcal{D}^{\sharp,h}_T(s_{\gamma},y)\geq \theta(s_{\gamma}).
\end{equation}

\ref{t:to0_limD}: Since $\theta\in \Gamma_0(X)$, we can apply \cite[Proposition 3.4.17]{ReginaBook2008} to conclude the existence of $u\in X^*$ and $\eta\in \RR$ such that $\theta\geq \langle \cdot, u \rangle + \eta$ (equivalently, $\dom \theta^*\neq \varnothing$). Using \eqref{e:gamma<mu} and Cauchy--Schwarz inequality, we have that, for all $\gamma\in \left]0,\mu\right]$,
\begin{subequations}
\begin{align}
\theta(y) &\geq \theta(s_{\gamma}) +\frac{1}{\gamma}\mathcal{D}^{\sharp,h}_T(s_{\gamma},y)\geq \langle s_{\gamma},u \rangle + \eta + \frac{1}{\gamma}\mathcal{D}^{\sharp,h}_T(s_{\gamma},y)\\
&\geq -\rho\|u\| + \eta + \frac{1}{\gamma}\mathcal{D}^{\sharp,h}_T(s_{\gamma},y),
\end{align}
\end{subequations}
which rearranges as 
\begin{equation}\label{e:D(z,y)}
0 \leq \mathcal{D}^{\sharp,h}_T(s_{\gamma},y)\leq \gamma \left(\theta(y) +\rho\|u\| -\eta \right) \rightarrow 0 \quad \text{as}\quad \gamma \downarrow 0.
\end{equation}
Therefore, $\mathcal{D}^{\sharp,h}_T(s_{\gamma},y) \rightarrow 0$ as $\gamma \downarrow 0$.

\ref{t:to0_D=0}: Let $(s_{k(\gamma)})$ be a subnet of $(s_\gamma)_{\gamma\in \left]0,\mu \right]}$ weakly converging to $z$ as $k(\gamma)\to 0$. Note that $\mathcal{D}^{\sharp,h}_T(\cdot,y)$ and $\theta$ are weakly lsc since they are convex and lsc. Applying \eqref{e:D(z,y)} to subnet $(s_{k(\gamma)})$ and using the weak lsc of $\mathcal{D}^{\sharp,h}_T(\cdot,y)$, we derive that
\begin{equation}\label{eq-v}
0 \leq \mathcal{D}^{\sharp,h}_T(z,y)\leq \liminf_{k(\gamma)\downarrow 0} \mathcal{D}^{\sharp,h}_T(s_{k(\gamma)},y)\leq \liminf_{k(\gamma)\downarrow 0} k(\gamma)(\theta(y) +\rho\|u\| -\eta) =0,
\end{equation}
which yields $\mathcal{D}^{\sharp,h}_T(z,y) =0$. Therefore, $Ty\subseteq Sz$ due to the definition of $\mathcal{D}^{\sharp,h}_T$ and property \ref{h:c} of $h$.  

Next, by applying the first inequality in \eqref{e:gamma<mu} to subnet $(s_{k(\gamma)})$ and using the weak lsc of $\theta$,
\begin{equation}
\theta(y)\geq \liminf_{k(\gamma)\downarrow 0} \left( \theta(s_{k(\gamma)}) +\frac{1}{k(\gamma)}\mathcal{D}^{\sharp,h}_T(s_{k(\gamma)},y)\right)\geq \liminf_{k(\gamma)\downarrow 0}  \theta(s_{k(\gamma)})\geq \theta(z),
\end{equation}
and so $z\in \lev{\theta(y)}\theta$.

\ref{t:to0_lim}: Let $z$ be an arbitrary weak cluster point $z$ of $(s_\gamma)_{\gamma\in \left]0,\mu \right]}$ as $\gamma\downarrow 0$. By assumption and \ref{t:to0_D=0}, $Ty\subseteq Tz$. Since $y\in \dom T$, taking $v\in Ty\subseteq Tz$, we have that $0=\langle z-y, v -v \rangle$ with $v\in Tz$ and $v\in Ty$. The strict monotonicity of $T$ implies that $z=y$. We deduce that $s_\gamma\rightharpoonup y$ as $\gamma\downarrow 0$. Now, recall from Proposition~\ref{p:inequalities}\ref{p:inequalities_left} that $\benv[\star,h]{\gamma,\theta}(y)\uparrow \beta$ as $\gamma\downarrow 0$ for some $\beta\in \mathbb{R}$. Using \eqref{e:gamma<mu} and the weak lsc of $\theta$, we obtain that
\begin{equation}\label{e:gamma<mu3}
\theta(y)\geq \beta\geq \liminf_{\gamma\downarrow 0} \theta(s_{\gamma})\geq \theta(y)\geq \limsup_{\gamma\downarrow 0} \theta(s_{\gamma}),
\end{equation}
which yields $\beta =\theta(y) =\lim_{\gamma\downarrow 0} \theta(s_{\gamma})$. Again using \eqref{e:gamma<mu}, this implies $\frac{1}{\gamma}\mathcal{D}^{\sharp,h}_T(s_{\gamma},y)\to 0$, and we are done. 
\end{proof}

\begin{theorem}[Asymptotic right behaviour when $\gamma \downarrow 0$]\label{t:Rto0}
Let $\theta\colon X\to \left]-\infty, +\infty\right]$ be proper and lsc, and let $x\in \dom S\cap \dom\theta$. Suppose that $X$ is finite-dimensional, that $\dom T\cap \dom\theta\neq \varnothing$, that $\{x\}\times (\dom T\cap \dom\theta)\subseteq \dom \mathcal{D}^{\star,h}_T$, that $\mathcal{D}^{\star,h}_T(x,\cdot)$ is lsc, and that $\varphi(\cdot) :=\theta(\cdot) +\frac{1}{\mu}\mathcal{D}^{\star,h}_T(x,\cdot)$ is coercive for some $\mu\in \RR_{++}$. For each $\gamma\in \left]0,\mu \right]$, let $s_\gamma :=\fproxs{\gamma,\theta}^{\star,h}(x)\in \fprox{\gamma,\theta}^{\star,h}(x)$. Then the following hold:
\begin{enumerate}
\item\label{t:to0_limD_right}
$\mathcal{D}^{\star,h}_T(x,s_{\gamma})\to 0$ as $\gamma\downarrow 0$.

\item\label{t:to0_D=0_right} 
For every cluster point $z$ of $(s_\gamma)_{\gamma\in \left]0,\mu \right]}$ as $\gamma\downarrow 0$, it holds that $\mathcal{D}^{\star,h}_T(x,z)=0$ and $z\in \lev{\theta(x)}\theta$.

\item\label{t:to0_lim_right} 
If $T=S$ and $S$ is strictly monotone over $\dom S\cap \dom \theta$, then, as $\gamma\downarrow 0$,
\begin{equation}
s_\gamma\to x,\quad \benv[\star,h]{\gamma,\theta}(x)\uparrow \theta(x),\quad \theta(s_{\gamma})\to \theta(x), \quad\text{and}\quad \frac{1}{\gamma}\mathcal{D}^{\star,h}_T(x,s_{\gamma})\to 0.    
\end{equation}
\end{enumerate}
\end{theorem}
\begin{proof}
This is proved similarly to Theorem~\ref{t:to0} by using Proposition~\ref{p:inequalities}\ref{p:inequalities_right}.
\end{proof}

It is worthwhile to connect these asymptotic results with previous ones in the literature. When $T=S=\nabla f$ for a strictly convex and differentiable $f:\RR^j \rightarrow \RR_\infty$, then $(Tz \subseteq Sz) \iff Tz=Sz$ and $(D_f(z,y)=0) \iff z=y$. In this specific case, Theorems~\ref{t:to0} and \ref{t:Rto0} show  both \cite[Proposition 3.2]{BDL17} and \cite[Theorem 3.3]{BDL17}\footnote{In \cite{BDL17}, the authors assume joint convexity and coercivity of $\mathcal{D}_f$ to obtain non-emptiness of the right proximity operator images by \cite[Proposition~3.5]{BCN06}; the latter result relies on the lower semicontinuity of the right distance as shown in \cite[Lemma~2.6]{BCN06}. Thus, our rather weak assumption that the right distance be lower semicontinuous is much less restrictive than the assumptions in \cite{BDL17}.}. Of course, Theorems~\ref{t:to0} and \ref{t:Rto0} are stronger. In addition to not requiring right convexity of the distance, these results also include envelopes that are not classical Bregman envelopes. We include several such examples of non-classical left envelopes in Figure~\ref{fig:entropy_left} and of non-classical right envelopes in Figure~\ref{fig:entropy_right}.

\begin{theorem}[Asymptotic left behaviour when $\gamma\uparrow +\infty$]
\label{t:toInfty}
Let $\theta\colon X\to \left]-\infty, +\infty\right]$, $y\in \dom T$, and $\gamma\in \mathbb{R}_{++}$. Suppose that $\dom S\cap \dom\theta\neq \varnothing$, that $(\dom S\cap \dom\theta)\times \{y\}\subseteq \dom \mathcal{D}^{\star,h}_T$, and that $Tz\subseteq Sz$ for all $z\in \dom T$. Then the following hold:
\begin{enumerate}
\item\label{t:toInfty_env}
$\benv[\star,h]{\gamma,\theta}(y)\downarrow \inf\theta(\dom S)$ as $\gamma\uparrow +\infty$. Consequently, if $\inf\theta(\overline{\dom}\, S) =\inf\theta(\dom S)$, then $\benv[\star,\dagger h]{\gamma,\theta}(y)\downarrow \inf\theta(\dom S)$ 
as $\gamma\uparrow +\infty$.

\item\label{t:toInfty_gen}
Suppose that $\theta\in \Gamma_0(X)$ and $\theta$ is coercive. For each $\gamma\in \mathbb{R}_{++}$, let $s_\gamma :=\bproxs{\gamma,\theta}^{\sharp,h}(y)\in \bprox{\gamma,\theta}^{\sharp,h}(y)$. Then 
\begin{equation}\label{e:toinfty}
\theta(s_\gamma)\to \inf\theta(\dom S) \quad\text{as}\quad \gamma\uparrow +\infty.    
\end{equation}
Moreover, if $\inf\theta(\overline{\dom}^w S) =\inf\theta(\dom S)$, then all weak cluster points of $(s_\gamma)_{\gamma\in \mathbb{R}_{++}}$ as $\gamma\uparrow +\infty$ lie in $\argmin\theta(\overline{\dom}^w S)\neq \varnothing$. If additionally $\argmin\theta(\overline{\dom}^w S)$ is a singleton, then $s_\gamma\to \argmin\theta(\overline{\dom}^w S)$ as $\gamma\uparrow +\infty$.

\item\label{t:toInfty_closed} 
Suppose that $\theta\in \Gamma_0(X)$, $\theta$ is coercive, and $\inf\theta(\overline{\dom}\, S) =\inf\theta(\dom S)$. For each $\gamma\in \mathbb{R}_{++}$, let $s_\gamma :=\bproxs{\gamma,\theta}^{\sharp,\dagger h}(y)\in \bprox{\gamma,\theta}^{\sharp,\dagger h}(y)$. Then 
\begin{equation}\label{e:toinftyb}
\theta(s_\gamma)\to \inf\theta(\dom S) \quad\text{as}\quad \gamma\uparrow +\infty.    
\end{equation}
Moreover, if $\inf\theta(\overline{\dom}^w S) =\inf\theta(\dom S)$, then all weak cluster points of $(s_\gamma)_{\gamma\in \mathbb{R}_{++}}$ as $\gamma\uparrow +\infty$ lie in $\argmin\theta(\overline{\dom}^w S)\neq \varnothing$. If additionally $\argmin\theta(\overline{\dom}^w S)$ is a singleton, then $s_\gamma\to \argmin\theta(\overline{\dom}^w S)$ as $\gamma\uparrow +\infty$.
\end{enumerate}
\end{theorem}
\begin{proof}
\ref{t:toInfty_env}: In view of Proposition~\ref{p:inequalities}\ref{p:inequalities_left}, 
\begin{equation}\label{e:lim env}
\benv[\star,h]{\gamma,\theta}(y)\downarrow \alpha\geq \inf\theta(\dom S) \quad\text{as}\quad \gamma\uparrow +\infty.
\end{equation}
By definition, for all $x\in \dom S\cap \dom\theta$,  
\begin{equation}
\benv[\star,h]{\gamma,\theta}(y) \leq \theta(x) +\frac{1}{\gamma}\mathcal{D}^{\star,h}_T(x,y),
\end{equation}
Letting $\gamma\uparrow +\infty$ and using the assumption that $(\dom S\cap \dom\theta)\times \{y\}\subseteq \dom \mathcal{D}^{\star,h}_T$, we derive that, for all $x\in \dom S\cap \dom\theta$, $\alpha\leq \theta(x)$, and so $\alpha\leq \inf\theta(\dom S\cap \dom\theta) =\inf\theta(\dom S)$. Combining with \eqref{e:lim env} implies that $\benv[\star,h]{\gamma,\theta}(y)\downarrow \inf\theta(\dom S)$ as $\gamma\uparrow +\infty$. In turn, by invoking \eqref{p:inequalities eq1'}, we get the second conclusion.

\ref{t:toInfty_gen}: According to Proposition~\ref{p:left prox}\ref{p:left prox_theta}, for all $\gamma\in \RR_{++}$, $\varnothing\neq \bprox{\gamma,\theta}^{\sharp,h}(y)\subseteq \dom S\cap \dom\theta$. Since $s_\gamma\in \bprox{\gamma,\theta}^{\sharp,h}(y)$, we have that $s_\gamma\in \dom S\cap \dom\theta$, and so
\begin{equation}
\inf\theta(\dom S)\leq \theta(s_\gamma)\leq \theta(s_\gamma) +\frac{1}{\gamma}\mathcal{D}^{\sharp,h}_T(s_\gamma,y) =\benv[\sharp,h]{\gamma,\theta}(y).
\end{equation}
By combining with \ref{t:toInfty_env}, $\theta(s_\gamma)\to \inf\theta(\dom S)$ as $\gamma\uparrow +\infty$.

Now, if $\inf\theta(\overline{\dom}^w S) =\inf\theta(\dom S)$, then we also have that $\theta(s_\gamma)\to \inf\theta(\overline{\dom}^w S)$ as $\gamma\uparrow +\infty$. The conclusion follows from \cite[Lemma~3.1]{BDL17}.  

\ref{t:toInfty_closed}: Similar to Proposition~\ref{p:left prox}\ref{p:left prox_theta}, we have that, for all $\gamma\in \RR_{++}$, $\varnothing\neq \bprox{\gamma,\theta}^{\sharp,\dagger h}(y)\subseteq \overline{\dom}\, S\cap \dom\theta$. Thus, $s_\gamma\in \overline{\dom}\, S\cap \dom\theta$ and
\begin{equation}
\inf\theta(\overline{\dom}\, S)\leq \theta(s_\gamma)\leq \theta(s_\gamma) +\frac{1}{\gamma}\overline{\mathcal{D}}^{\sharp,h}_T(s_\gamma,y) =\benv[\sharp,\dagger h]{\gamma,\theta}(y).
\end{equation} 
By assumption and \ref{t:toInfty_env}, $\benv[\sharp,\dagger h]{\gamma,\theta}(y)\to \inf\theta(\dom S) =\inf\theta(\overline{\dom}\, S)$ as $\gamma\uparrow +\infty$, and hence $\theta(s_\gamma)\to \inf\theta(\dom S)$ as $\gamma\uparrow +\infty$. Finally, proceeding as in \ref{t:toInfty_gen}, we complete the proof.
\end{proof}

\begin{theorem}[Asymptotic right behaviour when $\gamma\uparrow +\infty$]
\label{t:RtoInfty}
Let $\theta\colon X\to \left]-\infty, +\infty\right]$, $x\in \dom S$, and $\gamma\in \mathbb{R}_{++}$. Suppose that $\dom T\cap \dom\theta\neq \varnothing$, that $\{x\}\times (\dom T\cap \dom\theta)\subseteq \dom \mathcal{D}^{\star,h}_T$, and that $Tz\subseteq Sz$ for all $z\in \dom S$. Then the following hold:
\begin{enumerate}
\item\label{t:RtoInfty_env}
$\fenv[\star,h]{\gamma,\theta}(x)\downarrow \inf\theta(\dom T)$ as $\gamma\uparrow +\infty$. Consequently, if $\inf\theta(\overline{\dom}\, T) =\inf\theta(\dom T)$, then $\fenv[\star,\dagger h]{\gamma,\theta}(x)\downarrow \inf\theta(\dom T)$ 
as $\gamma\uparrow +\infty$.

\item\label{t:RtoInfty_gen}
Suppose that $X$ is finite-dimensional, that $\theta$ is lsc and coercive, and that $\mathcal{D}^{\star,h}_T(x,\cdot)$ is lsc. For each $\gamma\in \mathbb{R}_{++}$, let $s_\gamma :=\fproxs{\gamma,\theta}^{\star,h}(x)\in \fprox{\gamma,\theta}^{\star,h}(x)$. Then 
\begin{equation}\label{e:toinftyc}
\theta(s_\gamma)\to \inf\theta(\dom T) \quad\text{as}\quad \gamma\uparrow +\infty.    
\end{equation}
Moreover, if $\inf\theta(\overline{\dom}\, T) =\inf\theta(\dom T)$, then all cluster points of $(s_\gamma)_{\gamma\in \mathbb{R}_{++}}$ as $\gamma\uparrow +\infty$ lie in $\argmin\theta(\overline{\dom}\, T)\neq \varnothing$. If additionally $\argmin\theta(\overline{\dom}\, T)$ is a singleton, then $s_\gamma\to \argmin\theta(\overline{\dom}\, T)$ as $\gamma\uparrow +\infty$.

\item\label{t:RtoInfty_closed} 
Suppose that $X$ is finite-dimensional, that $\theta$ is lsc and coercive, and that $\inf\theta(\overline{\dom}\, T) =\inf\theta(\dom T)$. For each $\gamma\in \mathbb{R}_{++}$, let $s_\gamma :=\fproxs{\gamma,\theta}^{\star,\dagger h}(x)\in \fprox{\gamma,\theta}^{\star,\dagger h}(x)$. Then 
\begin{equation}\label{e:toinftyd}
\theta(s_\gamma)\to \inf\theta(\dom T) \quad\text{as}\quad \gamma\uparrow +\infty.    
\end{equation}
Moreover, all cluster points of $(s_\gamma)_{\gamma\in \mathbb{R}_{++}}$ as $\gamma\uparrow +\infty$ lie in $\argmin\theta(\overline{\dom}\, T)\neq \varnothing$. If additionally $\argmin\theta(\overline{\dom}\, T)$ is a singleton, then $s_\gamma\to \argmin\theta(\overline{\dom}\, T)$ as $\gamma\uparrow +\infty$.
\end{enumerate}
\end{theorem}
\begin{proof}
This is analogous to the proof of Theorem~\ref{t:toInfty} and uses Proposition~\ref{p:inequalities}\ref{p:inequalities_right} and Proposition~\ref{p:right prox}\ref{p:right prox_theta}. We note for \ref{t:RtoInfty_closed} that $\overline{\mathcal{D}}^{\star,h}_T(x,\cdot)$ is lsc.
\end{proof}

\begin{remark}
\label{r:closed}
We note that the condition $\inf\theta(\overline{\dom}^w S) =\inf\theta(\dom S)$ in Theorem~\ref{t:toInfty} is satisfied as soon as either 
\begin{enumerate}
\item 
$\dom S$ is weakly closed, or
\item 
$\theta\in \Gamma_0(X)$, $\dom S$ is convex, and $\inte\dom S\cap \dom\theta\neq \varnothing$.
\end{enumerate}
Indeed, the former case is obvious, while the latter case follows from \cite[Proposition~11.1(iv)]{BC17} (whose proof is still valid in a Banach space). Analogous statements hold true for the finite-dimensional counterpart $\inf\theta(\overline{\dom}\, T) =\inf\theta(\dom T)$ in Theorem~\ref{t:RtoInfty}.
\end{remark}

It is worthwhile to remark on the importance of the domain conditions imposed in the left Theorem~\ref{t:toInfty}\ref{t:toInfty_env}--\ref{t:toInfty_gen}, and in the right Theorem~\ref{t:RtoInfty}\ref{t:RtoInfty_env}--\ref{t:RtoInfty_gen}. In particular, the distance $\mathcal{D}^{\dagger \sigma_{\rm log}}$ does not satisfy them for all $x,y \in \dom \log$, and we use it to derive envelopes for $\theta=|\cdot-1/2|$ in Figures~\ref{fig:entropy_left_sigma} and \ref{fig:entropy_right_sigma} that do not have the asymptotic properties \ref{t:toInfty_env}--\ref{t:toInfty_gen} for every $x,y \in \dom \log$. Of course, the domain conditions we have imposed in Proposition~\ref{p:basic}, Theorem~\ref{t:toInfty}, and Theorem~\ref{t:RtoInfty} are still quite broad in what they cover. Corollary~\ref{c:env-to-inf} will showcase how the corresponding asymptotic guarantees subsume those of \cite[Proposition~2.2]{BDL17}, while generalizing them to many other classes of distances and envelopes. In particular, we illustrate with the GBD $\mathcal{D}^{\dagger F_{\log}}$ by constructing its generalized left and right envelopes in Figures~\ref{fig:entropy_left_fitz} and \ref{fig:entropy_right_fitz}. It may be seen from the figures that these new envelopes respectively satisfy the asymptotic guarantees of Theorem~\ref{t:toInfty}\ref{t:toInfty_env}--\ref{t:toInfty_gen} and Theorem~\ref{t:RtoInfty}\ref{t:RtoInfty_env}--\ref{t:RtoInfty_gen}.

\begin{corollary}
\label{c:env-to-inf}
Let $f\in \Gamma_0(X)$ with $D :=\dom\partial f$. Suppose that $T =S =\partial f$ and that $h\in \mathcal{H}(\partial f)$ satisfies
\begin{equation}
\forall (x,v)\in \dom f\times \ran\partial f,\quad h(x,v)\leq (f\oplus f^\ast)(x,v) =f(x) +f^\ast(v).
\end{equation}
Let $\theta\colon X\to \RR_{\infty}$ with $D\cap \dom\theta \neq\varnothing$, let $\gamma\in \RR_{++}$, and let $x, y\in D$. Then the following hold:
\begin{enumerate}
\item\label{c:env-to-inf_flat} 
As $\gamma\uparrow +\infty$, 
\begin{enumerate}
\item\label{c:env-to-inf_flat_gen} $\benv[\flat,h]{\gamma,\theta}(y)\downarrow \inf\theta(D)$ and $\fenv[\flat,h_{}]{\gamma,\theta}(x)\downarrow \inf\theta(D)$;
\item\label{c:env-to-inf_flat_closed} $\benv[\flat,\dagger h]{\gamma,\theta}(y)\downarrow \inf\theta(D)$ and $\fenv[\flat,\dagger h]{\gamma,\theta}(x)\downarrow \inf\theta(D)$;
\item\label{c:env-to-inf_flat_Breg} Classic Bregman envelopes satisfy
$\benv[\flat]{\gamma,\theta}(y)\downarrow \inf\theta(D)$ and $\fenv[\flat]{\gamma,\theta}(x)\downarrow \inf\theta(D)$.
\end{enumerate}
\item\label{c:env-to-inf_sharp} 
Suppose further that $\dom \partial f$ is open. Then, as $\gamma\uparrow +\infty$, 
\begin{enumerate}
\item\label{c:env-to-inf_sharp_gen} $\benv[\sharp, h_{}]{\gamma,\theta}(y)\downarrow \inf\theta(D)$ and $\fenv[\sharp,h_{}]{\gamma,\theta}(x)\downarrow \inf\theta(D)$;
\item\label{c:env-to-inf_sharp_closed} $\benv[\sharp,\dagger h]{\gamma,\theta}(y)\downarrow \inf\theta(D)$ and $\fenv[\sharp,\dagger h]{\gamma,\theta}(x)\downarrow \inf\theta(D)$;
\item\label{c:env-to-inf_sharp_Breg} Classic Bregman envelopes satisfy
$\benv[\sharp]{\gamma,\theta}(y)\downarrow \inf\theta(D)$ and $\fenv[\sharp]{\gamma,\theta}(x)\downarrow \inf\theta(D)$.
\end{enumerate}
\end{enumerate}
\end{corollary}
\begin{proof}
\ref{c:env-to-inf_flat}: By Lemma~\ref{l:domD}\ref{l:domD_flat}, $\dom \mathcal{D}_{h}^\flat =D\times D$, so $\mathcal{D}_{h}^\flat$ satisfies Theorem~\ref{t:toInfty}\ref{t:toInfty_env} and Theorem~\ref{t:RtoInfty}\ref{t:RtoInfty_env}, and we therefore get \ref{c:env-to-inf_flat_gen} and \ref{c:env-to-inf_flat_closed}.

Now, since $y\in D =\dom\partial f$, in view of Proposition~\ref{p:fitz_bregman} and Definition~\ref{d:env&prox}, we have that $\benv[\star]{\gamma,\theta}(y) =\benv[\star,f\oplus f^\ast]{\gamma,\theta}(y)$ and $\fenv[\star]{\gamma,\theta}(x) =\fenv[\star,f\oplus f^\ast]{\gamma,\theta}(x)$. Thus \ref{c:env-to-inf_flat_Breg} follows by applying \ref{c:env-to-inf_flat_gen} with $h =f\oplus f^\ast$.

\ref{c:env-to-inf_sharp}: As $\dom \partial f$ is open, Lemma~\ref{l:domD}\ref{l:domD_sharp} implies that $\dom \mathcal{D}_{\partial f}^{\sharp,h} =D\times D$. The proof is then completed by a similar argument as in \ref{c:env-to-inf_flat}. 
\end{proof}

The results in Corollary~\ref{c:env-to-inf}\ref{c:env-to-inf_flat_Breg}\&\ref{c:env-to-inf_sharp_Breg} for classical Bregman distances reduce to the ones in \cite[Proposition~2.2]{BDL17} when $f$ is a differentiable convex function. The following corollary provides analogous extensions for the generalized proximity operators as well. We denote by $\bprox{\gamma,\theta}^{\star},\fprox{\gamma,\theta}^{\star}$ the classical Bregman proximity operators associated with the classical Bregman envelopes $\benv[\star]{\gamma,\theta},\fenv[\star]{\gamma,\theta}$.

\begin{corollary}
\label{c:toInfty}
Let $f\in \Gamma_0(X)$ with $D :=\dom\partial f$ open. Suppose that $T =S =\partial f$ and that $h\in \mathcal{H}(\partial f)$ satisfies
\begin{equation}
\forall (x,v)\in \dom f\times \ran\partial f,\quad h(x,v)\leq (f\oplus f^\ast)(x,v) =f(x) +f^\ast(v).
\end{equation}
Let $\theta\in \Gamma_0(X)$ be coercive with $\inte D\cap \dom\theta\neq \varnothing$, let $\gamma\in \RR_{++}$, and let $y\in D$. Suppose that one of the following holds:
\begin{enumerate}
\item\label{c:toInfty_genb}
For each $\gamma\in \mathbb{R}_{++}$, $s_\gamma :=\bproxs{\gamma,\theta}^{\sharp,h}(y)\in \bprox{\gamma,\theta}^{\sharp,h}(y)$;
\item\label{c:toInfty_closed} 
For each $\gamma\in \mathbb{R}_{++}$, $s_\gamma :=\bproxs{\gamma,\theta}^{\sharp,\dagger h}(y)\in \bprox{\gamma,\theta}^{\sharp,\dagger h}(y)$;
\item\label{c:toInfty_Breg} 
For each $\gamma\in \mathbb{R}_{++}$, $s_\gamma :=\bproxs{\gamma,\theta}^{\sharp}(y)\in \bprox{\gamma,\theta}^{\sharp}(y)$.
\end{enumerate}
Then $\theta(s_\gamma)\to \inf\theta(D)$ as $\gamma\uparrow +\infty$. Moreover, the net $(s_\gamma)_{\gamma \in \RR_{++}}$ is bounded with all cluster points as $\gamma\uparrow +\infty$ lying in $\argmin\theta(\overline{D}^w)\neq \varnothing$. If additionally $\argmin\theta(\overline{D}^w)$ is a singleton, then $s_\gamma\to \argmin\theta(\overline{D}^w)$ as $\gamma\uparrow +\infty$. 
\end{corollary}
\begin{proof}
We first have from Lemma~\ref{l:domD}\ref{l:domD_sharp} that $\dom \mathcal{D}_{\partial f}^{\sharp,h} =D\times D$. In view of Remark~\ref{r:closed}, $\inf\theta(\overline{D}^w) =\inf\theta(D)$. So, the conditions of Theorem~\ref{t:toInfty} are satisfied, and we have the desired result in cases \ref{c:toInfty_genb}, and \ref{c:toInfty_closed}.

As shown in the proof of Corollary~\ref{c:env-to-inf}\ref{c:env-to-inf_flat},  $\benv[\star]{\gamma,\theta}(y) =\benv[\star,f\oplus f^\ast]{\gamma,\theta}(y)$, which implies that $\bprox{\gamma,\theta}^{\star} =\bprox{\gamma,\theta}^{\star,f\oplus f^\ast}$. Therefore, the desired result in case \ref{c:toInfty_Breg} follows from \ref{c:toInfty_genb} with $h =f\oplus f^\ast$.
\end{proof}

\begin{corollary}
\label{c:RtoInfty}
Let $f\in \Gamma_0(X)$ with $D :=\dom\partial f$. Suppose that $T =S =\partial f$ and that $h\in \mathcal{H}(\partial f)$ satisfies
\begin{equation}
\forall (x,v)\in \dom f\times \ran\partial f,\quad h(x,v)\leq (f\oplus f^\ast)(x,v) =f(x) +f^\ast(v).
\end{equation}
Let $\theta\colon X\to \left]-\infty, +\infty\right]$ be lsc and coercive with $\inte D\cap \dom\theta\neq \varnothing$, let $\gamma\in \RR_{++}$, and let $x\in D$. Suppose that $X$ is finite-dimensional and that one of the following holds:
\begin{enumerate}
\item\label{c:toInfty_gproxes aenc}
$\mathcal{D}^{\flat,h}_T(x,\cdot)$ is lsc and, for each $\gamma\in \mathbb{R}_{++}$, $s_\gamma :=\fproxs{\gamma,\theta}^{\flat,h}(x)\in \fprox{\gamma,\theta}^{\flat,h}(x)$;
\item\label{c:toInfty_closedb} 
For each $\gamma\in \mathbb{R}_{++}$, $s_\gamma :=\fproxs{\gamma,\theta}^{\flat,\dagger h}(x)\in \fprox{\gamma,\theta}^{\flat,\dagger h}(x)$;
\item\label{c:toInfty_Bregb} 
$\mathcal{D}_{\partial f}^{\flat}(x,\cdot)$ is lsc and, for each $\gamma\in \mathbb{R}_{++}$, $s_\gamma :=\fproxs{\gamma,\theta}^{\flat}(x)\in \fprox{\gamma,\theta}^{\flat}(x)$.
\item 
$\dom\partial f$ is open, $\mathcal{D}^{\sharp,h}_T(x,\cdot)$ is lsc, and, for each $\gamma\in \mathbb{R}_{++}$, $s_\gamma :=\fproxs{\gamma,\theta}^{\sharp,h}(x)\in \fprox{\gamma,\theta}^{\sharp,h}(x)$;
\item\label{c:toInfty_closedc} 
$\dom\partial f$ is open and, for each $\gamma\in \mathbb{R}_{++}$, $s_\gamma :=\fproxs{\gamma,\theta}^{\sharp,\dagger h}(x)\in \fprox{\gamma,\theta}^{\sharp,\dagger h}(x)$;
\item\label{c:toInfty_Bregc} 
$\dom\partial f$ is open, $\mathcal{D}_{\partial f}^{\sharp}(x,\cdot)$ is lsc, and, for each $\gamma\in \mathbb{R}_{++}$, $s_\gamma :=\fproxs{\gamma,\theta}^{\sharp}(x)\in \fprox{\gamma,\theta}^{\sharp}(x)$.
\end{enumerate}
Then $\theta(s_\gamma)\to \inf\theta(D)$ as $\gamma\uparrow +\infty$. Moreover, the net $(s_\gamma)_{\gamma \in \RR_{++}}$ is bounded with all cluster points as $\gamma\uparrow +\infty$ lying in $\argmin\theta(\overline{D})\neq \varnothing$. If additionally $\argmin\theta(\overline{D})$ is a singleton, then $s_\gamma\to \argmin\theta(\overline{D})$ as $\gamma\uparrow +\infty$. 
\end{corollary}
\begin{proof}
The proof is similar to Corollary~\ref{c:toInfty}, using Lemma~\ref{l:domD} and Theorem~\ref{t:RtoInfty}.
\end{proof}

When $f$ is differentiable, the results for the classical Bregman envelopes given in Corollaries~\ref{c:toInfty} and \ref{c:RtoInfty} generalize those in \cite[Theorem~3.5]{BDL17}.

\subsection*{Examples of envelopes and proximity operators}\label{sec:examples}

Theorems~\ref{t:to0}--\ref{t:RtoInfty} and Corollaries~\ref{c:env-to-inf}, \ref{c:toInfty}, and \ref{c:RtoInfty} are important for several reasons. They reveal that some of the asymptotic results provided in \cite[Propositions~2.2 \& 3.2, Theorems~3.3 \& 3.5]{BDL17} for Bregman envelopes are only a special case of a class results that hold for envelopes constructed from representative functions for maximally monotone operators. We also note that we have also here provided a clarification of the domain conditions in the exposition of \cite[Propositions~2.2 \& 3.2]{BDL17}, namely, that the asymptotic results are, of course, restricted to the domain of the distance operator.

\begin{figure}
\begin{center}
\subfloat[{$\protect\benv[\dagger F_{\log}]{\gamma,\theta}$}\label{fig:entropy_left_fitz}]{\trimplot{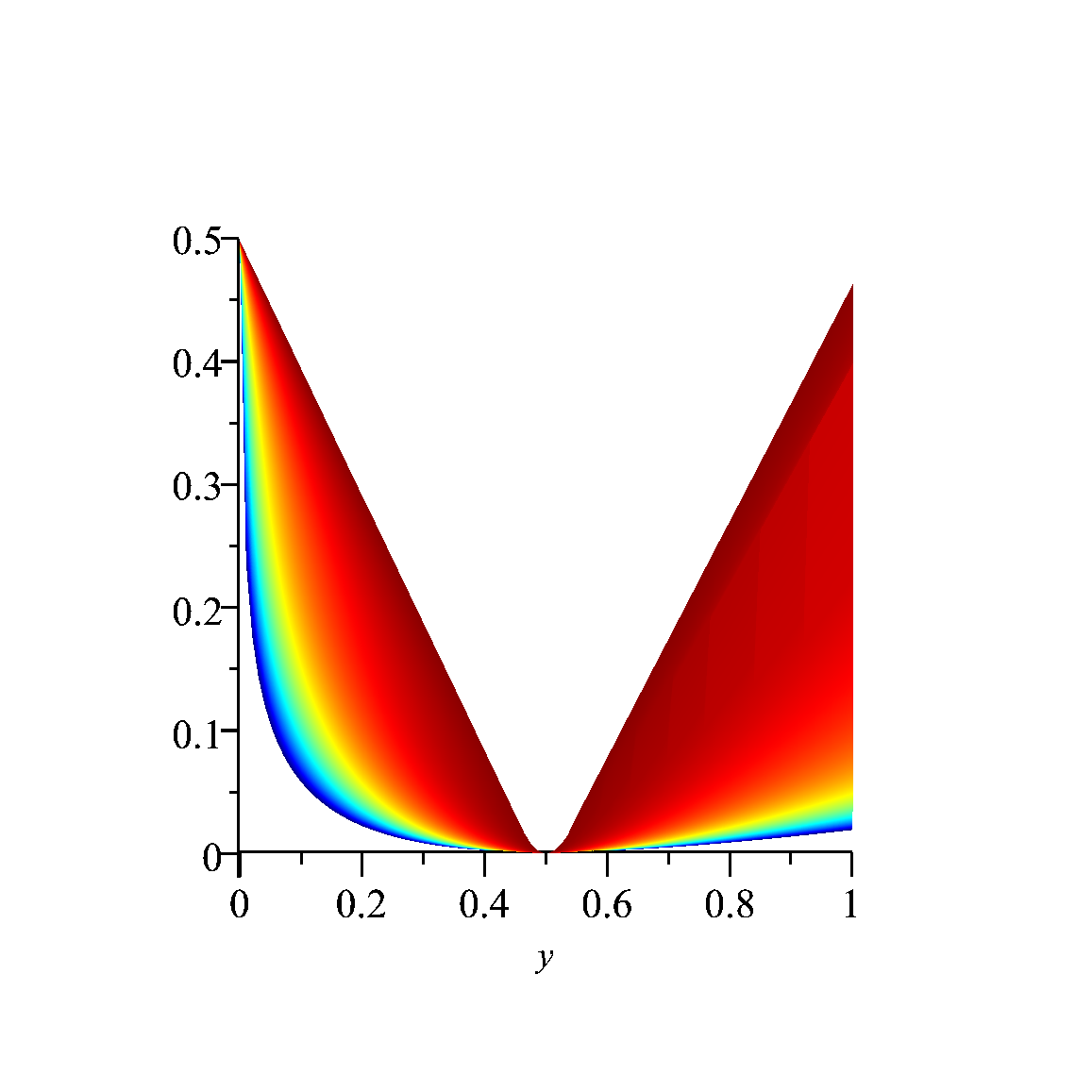}}
\subfloat[{$\protect\benv[\dagger \ent \oplus \ent^*]{\gamma,\theta}$}\label{fig:entropy_left_bregman}]{\trimplot{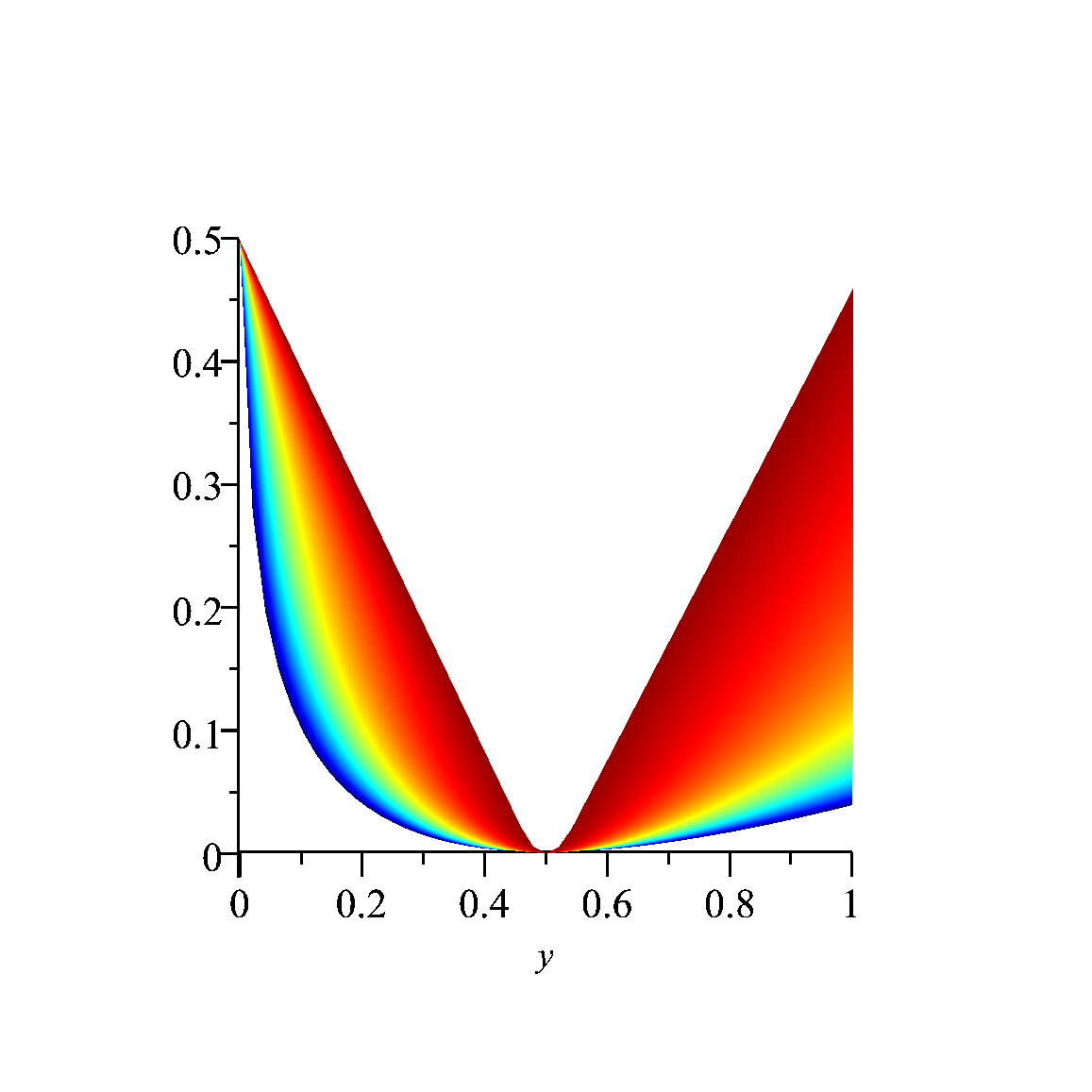}}
\subfloat[{$\protect\benv[\dagger \sigma_{\log}]{\gamma,\theta}$}\label{fig:entropy_left_sigma}]{\trimplot{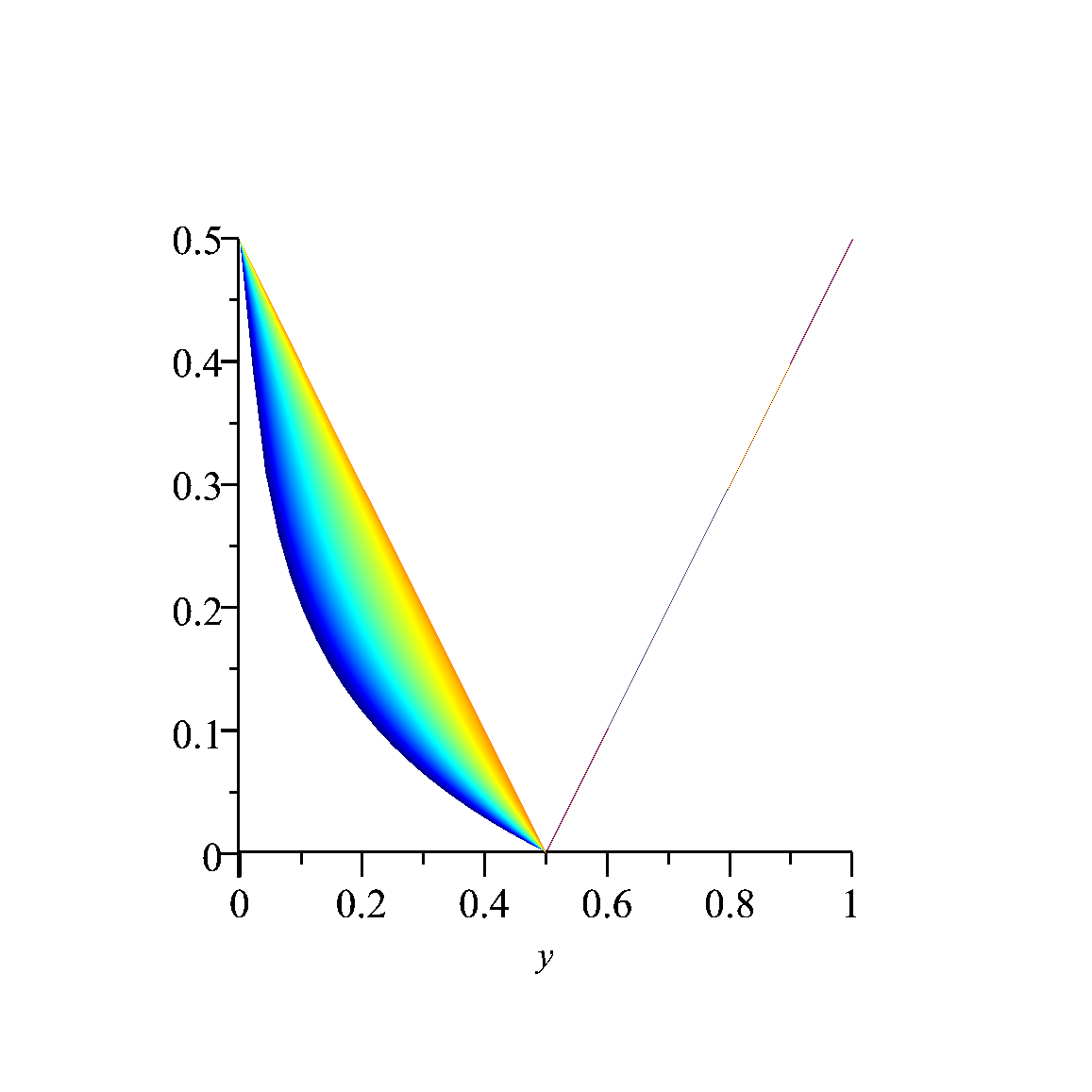}}
\subfloat[key]{\begin{adjustbox}{trim={0.35\width} {0.16\height} {0.4\width} {0.16\height},clip=true}
\includegraphics[width=.4\textwidth]{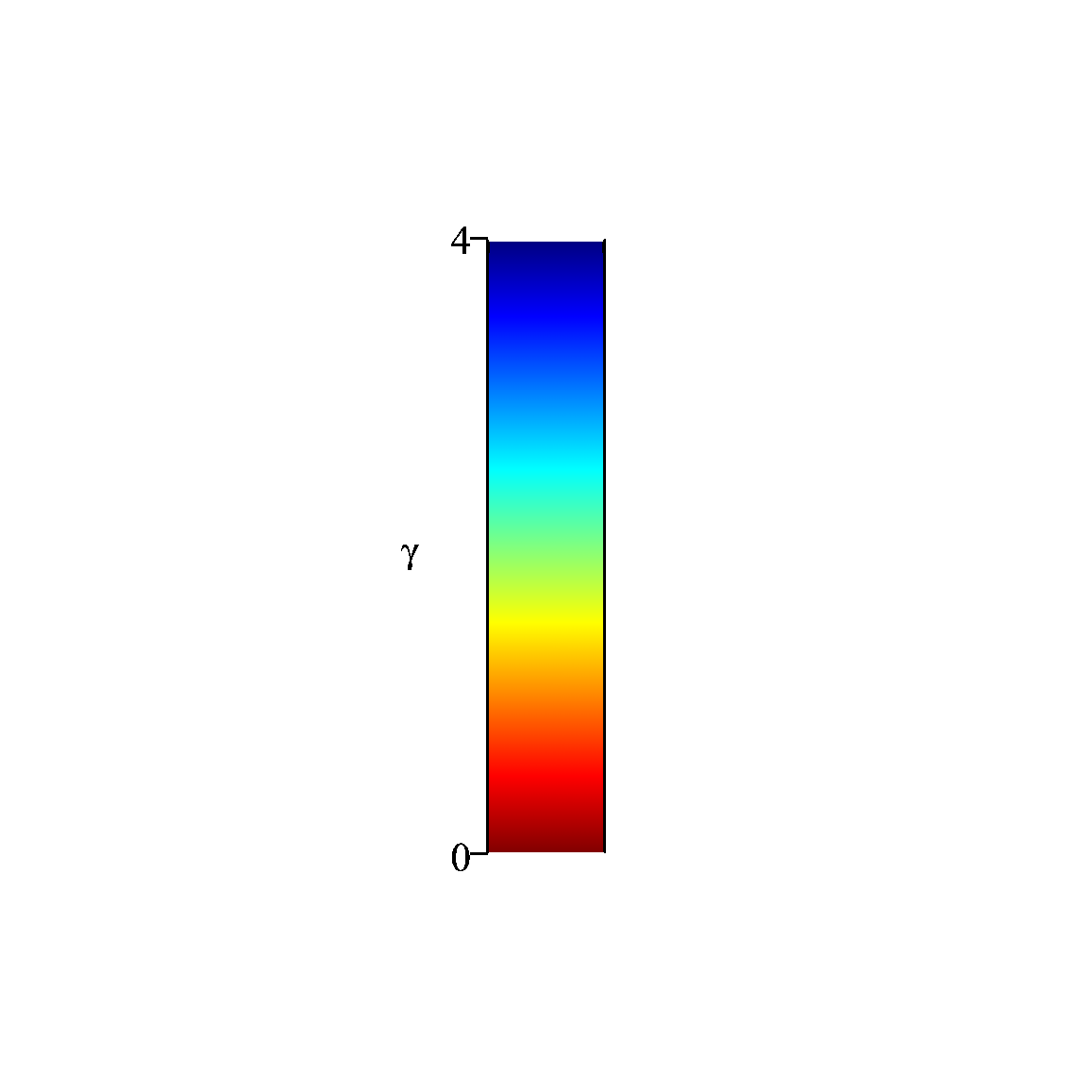}
\end{adjustbox}}
\end{center}
\caption{Left envelopes for representative functions of the logarithm.}\label{fig:entropy_left}
\end{figure}

We illustrate these connections with the envelopes that correspond to the distances $\overline{\mathcal{D}}^{F_{\log}}$, $\overline{\mathcal{D}}^{\ent \oplus \ent^*}$, and $\overline{\mathcal{D}}^{\sigma_{\log}}$ from Examples~\ref{ex:Fitzlog}, \ref{ex:KullbackLiebler}, and \ref{ex:sigmalog} respectively, which are illustrated in Figure~\ref{fig:D}. The derivation of these distances may be found in \cite{BDL19}. The left envelopes are shown in Figure~\ref{fig:entropy_left}, and the right envelopes are shown in Figure~\ref{fig:entropy_right}. For both figures, we use the closed versions of the distances. The explicit forms for the envelopes and their proximity operators are given in Appendix~\ref{sec:appendix}.

\begin{figure}
\begin{center}
\subfloat[{$\protect\fenv[\dagger F_{\log}]{\gamma,\theta}$}\label{fig:entropy_right_fitz}]{\trimplot{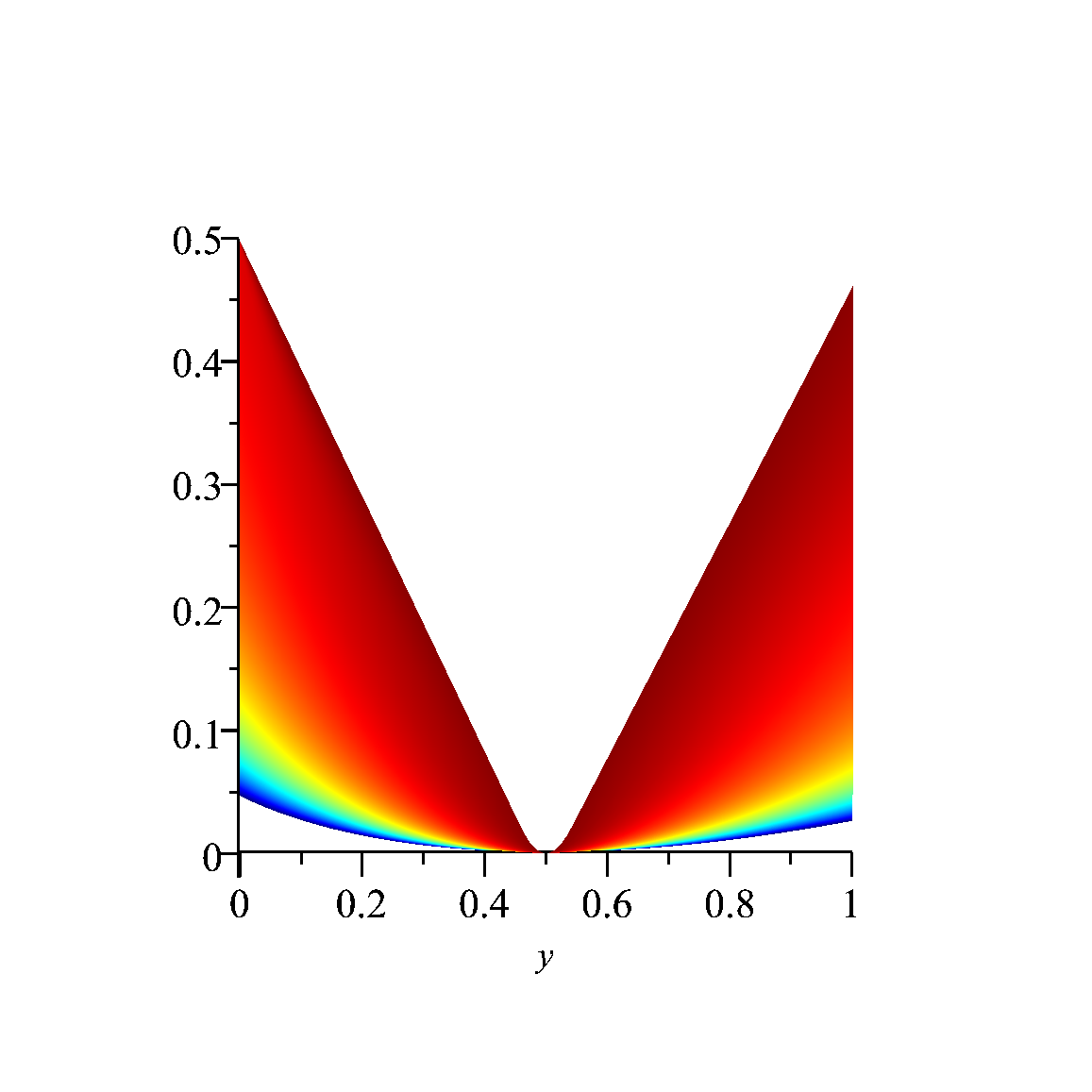}}
\subfloat[{$\protect\fenv[\dagger \ent \oplus \ent^*]{\gamma,\theta}$}\label{fig:entropy_right_bregman}]{\trimplot{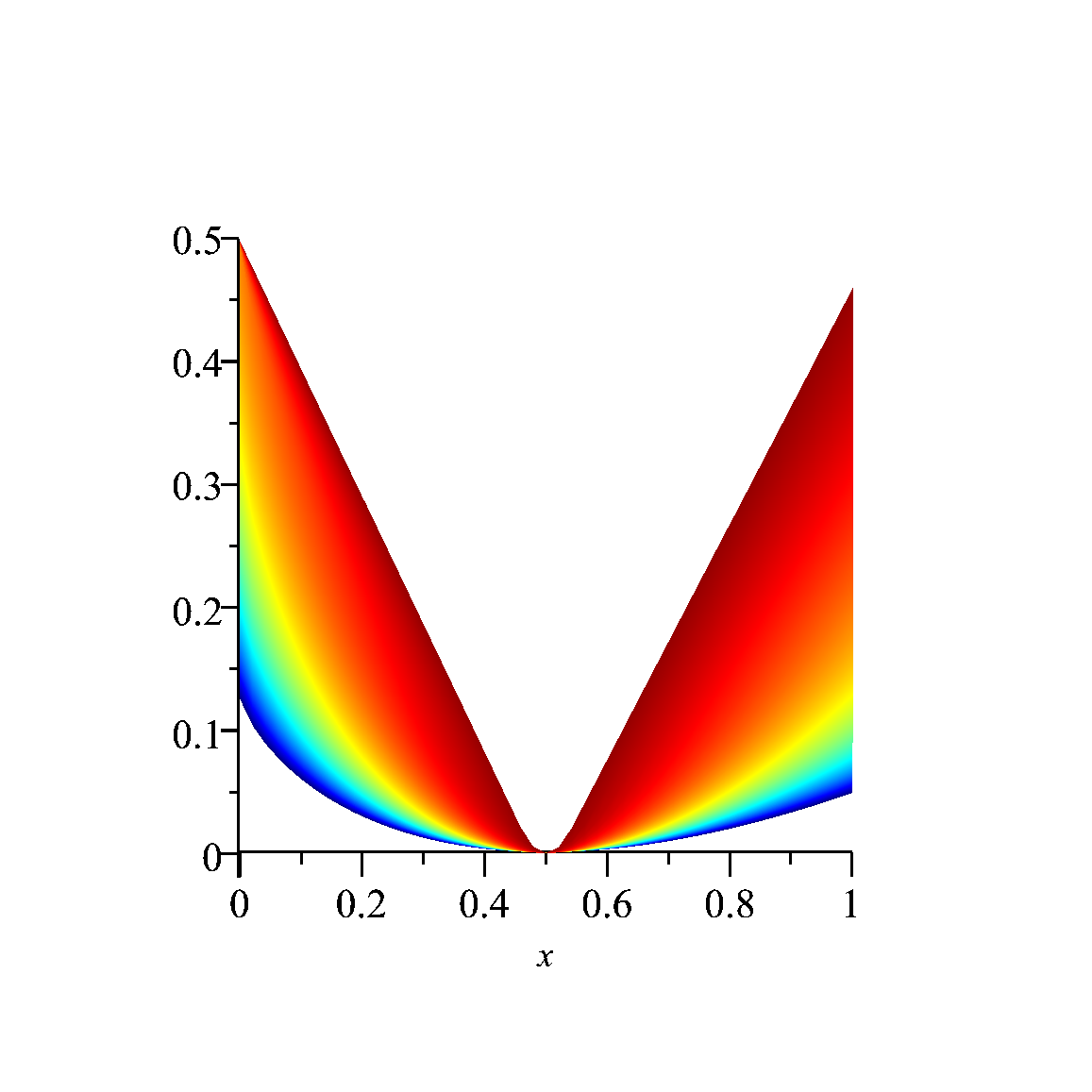}}
\subfloat[{$\protect\fenv[\dagger \sigma_{\log}]{\gamma,\theta}$}\label{fig:entropy_right_sigma}]{\trimplot{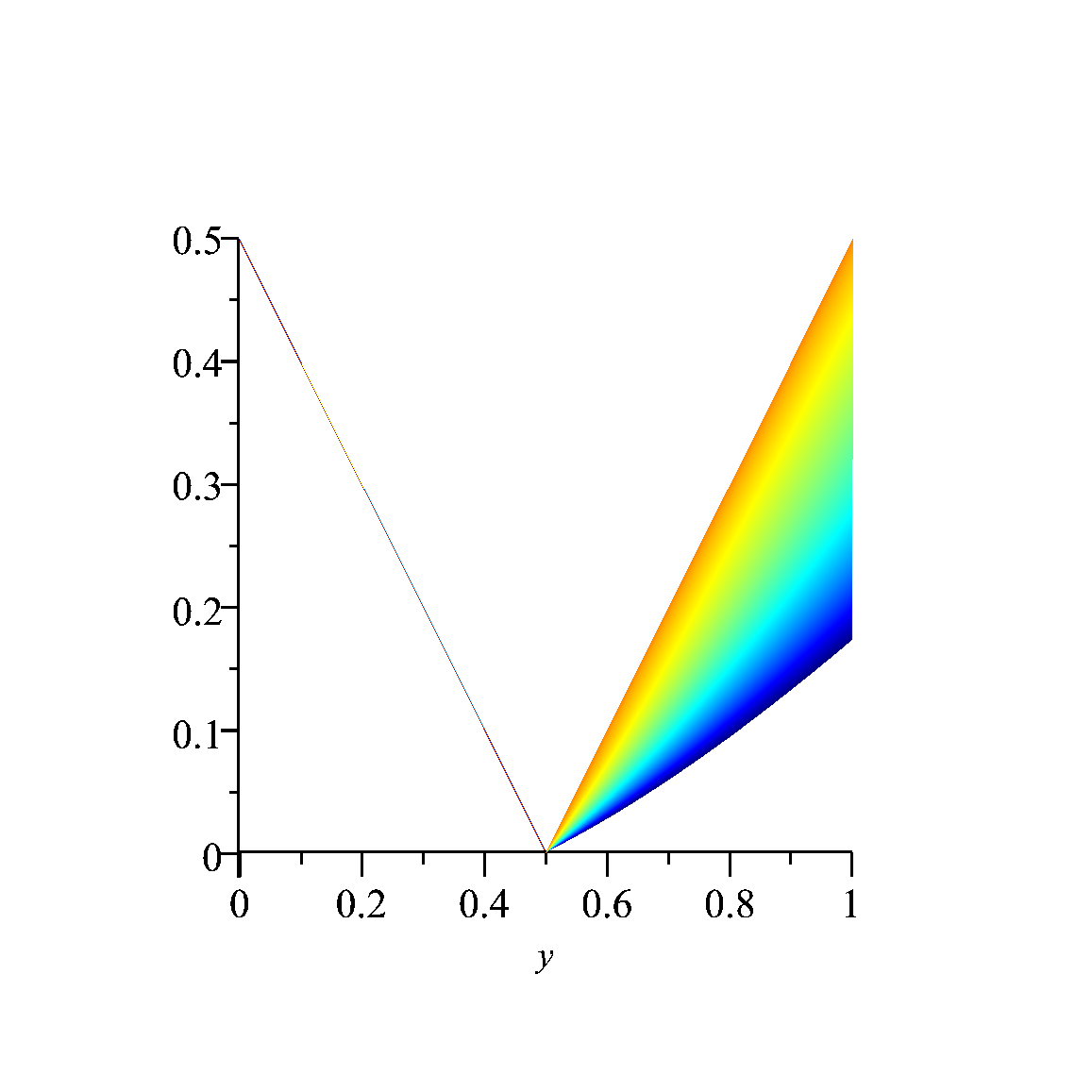}}
\subfloat[key]{\begin{adjustbox}{trim={0.35\width} {0.16\height} {0.4\width} {0.16\height},clip=true}
\includegraphics[width=.4\textwidth]{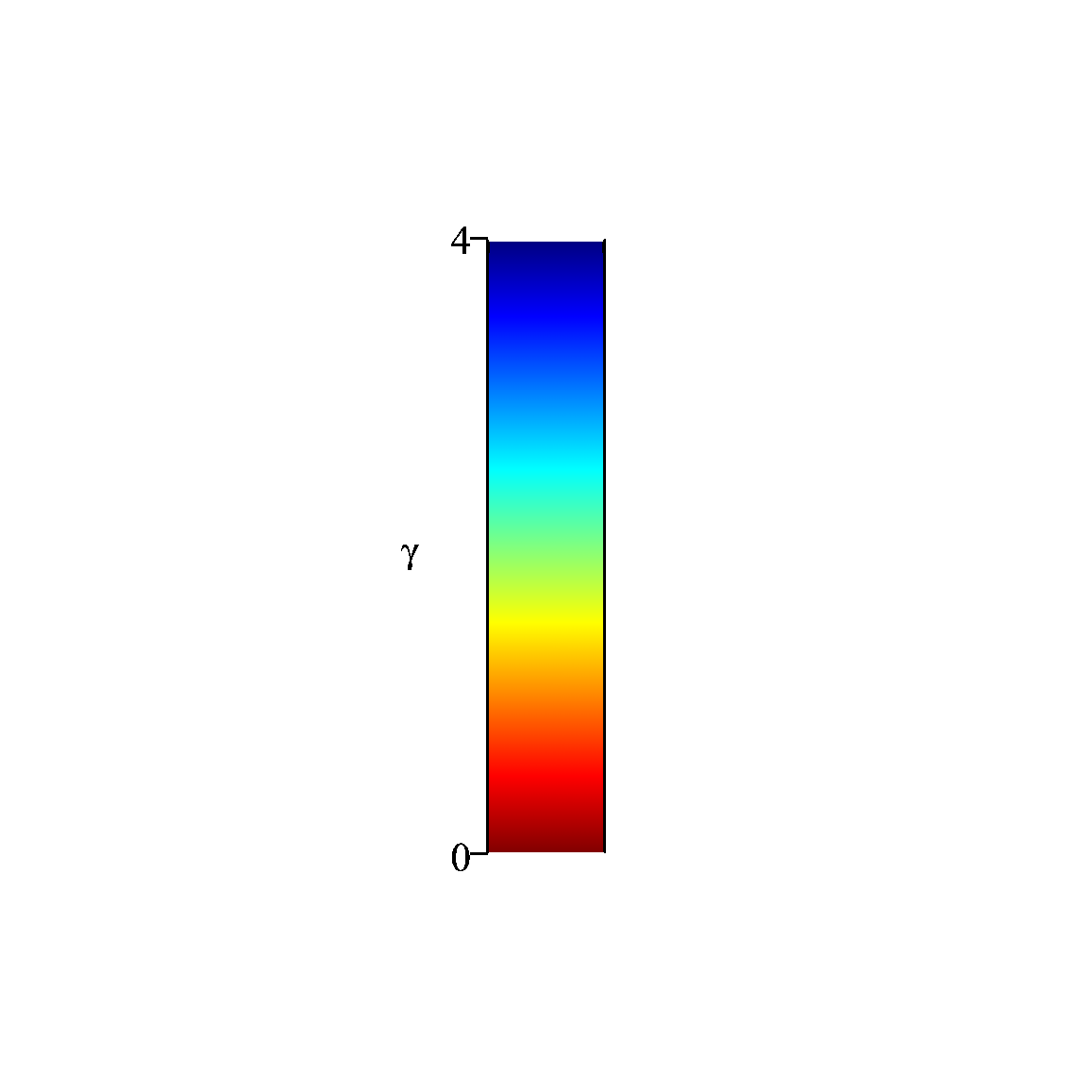}
\end{adjustbox}}
\end{center}
\caption{Right envelopes for representative functions of the logarithm.}\label{fig:entropy_right}
\end{figure}

Let us first discuss how these examples illustrate Theorems~\ref{t:to0} \& \ref{t:Rto0}. As the parameter $\gamma \to 0$, all of the envelopes in Figures~\ref{fig:entropy_left} and Figure~\ref{fig:entropy_right} exhibit the behaviour of approaching $\theta$, the Legendre function being regularized. This is, at first, less clear when the representative function employed is $\sigma_{\log}$ as in Figures~\ref{fig:entropy_left_sigma} and \ref{fig:entropy_right_sigma}, and so some clarification is in order. For any $\gamma \in \left]0,+\infty \right[$, the function $\protect\benv[\dagger \sigma_{\log}]{\gamma,\theta}$ is exactly equal to $\theta$ on $\left[0,1/2 \right]$, while the function $\protect\fenv[\dagger \sigma_{\log}]{\gamma,\theta}$ is exactly equal to $\theta$ on $\left[1/2,1 \right]$. This is why the net of curves collapses to a single line segment on $\left[0,1/2 \right]$ in Figure~\ref{fig:entropy_left_sigma} and on $\left[1/2,1 \right]$ in Figure~\ref{fig:entropy_right_sigma}. The reason for this may be found by scrutinizing the distance $\overline{\mathcal{D}}^{\sigma_{\log}}$ in \eqref{eqn:Dsigma} and in Figure~\ref{D:sigma}, and observing that the distance takes the value infinity whenever the right variable is greater than the left. For this reason, the sum $\theta(\cdot)+\overline{\mathcal{D}}^{\sigma_{\log}}(\cdot,y)$ is minimized at $y$ for $y\geq 1/2$, while the sum $\theta(\cdot)+\overline{\mathcal{D}}^{\sigma_{\log}}(x,\cdot)$ is minimized at $x$ for $x\leq 1/2$.

The case of $\overline{\mathcal{D}}^{\sigma_{\log}}$ is also instrumental in understanding Theorems~\ref{t:toInfty} \& \ref{t:RtoInfty} and Corollary~\ref{c:env-to-inf}. For the reasons we have just discussed, the condition $\theta(\bproxs{\gamma,\theta}^{\dagger \sigma_{\log}}(y))\to \inf\theta(U)$ as $\gamma\uparrow +\infty$ fails to hold for $y \in \left]1/2,1 \right]$. Consequently, for $y \in \left]1/2,1\right]$ the condition $\benv[\dagger \sigma_{\log}]{\gamma,\theta}(y)\downarrow \inf\theta(U)$ as $\gamma\uparrow +\infty$ does not hold. An analogous situation arises for $\fenv[\dagger \sigma_{\log}]{\gamma,\theta}(x)$ in the case of $x \in \left[0,1/2\right[$. As previously mentioned in Examples~\ref{ex:dist:energy} and \ref{ex:env:energy}, $\mathcal{D}^{\sigma_{\Id}}=\iota_{\mathcal{G}_{\Id}}$ is an even more dramatic case where the GBD envelope does not asymptotically approach $\inf \theta(U)$.

On the other hand, for $y \in \left[0,1/2 \right]$ we have that $\theta(\bproxs{\gamma,\theta}^{\dagger \sigma_{\log}}(y))\to \inf\theta(U)$ and $\benv[\dagger \sigma_{\log}]{\gamma,\theta}(y)\downarrow \inf\theta(U)$ as $\gamma\uparrow +\infty$. Similarly, for $x \in \left[1/2,1 \right]$ we have that $\theta(\fproxs{\gamma,\theta}^{\dagger \sigma_{\log},\sigma_{\log}}(y))\to \inf\theta(U)$ and $\fenv[\dagger \sigma_{\log}]{\gamma,\theta}(y)\downarrow \inf\theta(U)$ as $\gamma\uparrow +\infty$. 

The loss of some desirable asymptotic properties for the largest member of $\mathcal{H}(\partial f)$ highlights the advantage of Corollary~\ref{c:toInfty}, which assures us that $\benv[\dagger F_{\log}]{\gamma,\theta}(y)\downarrow \inf\theta(U)$ and $\fenv[\dagger F_{\log}]{\gamma,\theta}(x)\downarrow \inf\theta(U)$ as $\gamma \uparrow +\infty$, because $F_{\log} \leq \ent \oplus \ent^*$. We see this property illustrated in Figures~\ref{fig:entropy_left_fitz} and \ref{fig:entropy_right_fitz}. Comparing with Figures~\ref{fig:entropy_left_bregman} and \ref{fig:entropy_right_bregman}, we can also see the essential property from the proof of Corollary~\ref{c:toInfty}: that envelopes built from smaller representative functions than the Fenchel--Young representative are majorized thereby.

Finally, from a theoretical standpoint, the case of $\fprox{\gamma,\theta}^{\dagger F_{\log}}$ is illustrative of why we had to employ selection operators in our analysis of the proximity operators, because it is set valued for $x=0$ and $e=1/\gamma$.

\section{Conclusion}\label{sec:conclusion}

In Section~\ref{sec:preliminaries}, we recalled the theory of GBDs \cite{BM-L18} and the coercivity framework established in \cite{BDL19}. We also introduced the specific computed distances from \cite{BDL19} that we have used to build the envelopes and proximity operators in the current exposition, along with an explanation of why they are natural distances to consider. In Section~\ref{sec:envelopes}, we introduced the left and right envelopes for the GBDs, along with their associated proximity operators. In Section~\ref{sec:asymptotic}, we provided a selection of asymptotic results. The examples in Section~\ref{sec:examples} illustrate how the results in the setting of GBDs vary from those we obtain more easily when specializing to Bregman distances. 

Our analysis also yields results on the Bregman case when specializing thereto by using the Fenchel--Young representative. Pleasingly, the desirable asymptotic properties for Bregman distances extend to GBDs constructed from Fitzpatrick representatives, which suggests that such distances may be a useful subject for specialized investigation. Many important optimization algorithms may be studied as special cases of gradient descent applied to envelopes; now that a sufficient coercivity framework has been developed \cite{BDL19} and distances with desirable asymptotic envelope properties identified, two natural follow-up questions present themselves. The first is: excluding the already known Bregman cases, are there useful descent algorithms---discovered or undiscovered---whose analysis may fit within such a framework? The second is: are there previously unknown GBDs---for example, GBDs constructed from Fitzpatrick representatives---whose forms admit computational advantages over their Bregman counterparts?

\subsection*{Acknowledgements}

The authors are grateful to the two anonymous referees for their careful comments and suggestions. Part of this work was done during MND's visit to the University of South Australia in 2018 to whom he acknowledges the hospitality. SBL was supported by an Australian Mathematical Society Lift-Off Fellowship, and by Hong Kong Research Grants Council PolyU153085/16p.

\appendix
\section{Appendix: Closed forms for envelopes and proximity operators}\label{sec:appendix}

For all our examples, we use the closed distances as described in Remark~\ref{r:lsc}. The function denoted $\W$ is the principal branch of the Lambert $\W$ function, whose occurrences in variational analysis have been discussed in, for example, \cite{BL2018,BL2017,Knuth}. 

\begin{example}[Left prox and envelope for $\overline{\mathcal{D}}^{F_{\log}}$]

Beginning with the smallest member of $\mathcal{H}(\log)$, we first consider the left envelope and corresponding proximity operator characterized by 
\begin{align*}
\benv[\dagger F_{\log}]{\gamma,\theta}(x)&= \underset{y \in \RR_+}{\inf} \{\theta(y)+\frac{1}{\gamma}\overline{\mathcal{D}}^{F_{\log}}(y,x)  \}\\
&=\theta(\bprox{\gamma,\theta}^{\dagger F_{\log}} (x))+\frac{1}{\gamma}\overline{\mathcal{D}}^{F_{\log}}\left(\bprox{\gamma,\theta}^{\dagger F_{\log}} (x),x \right),
\end{align*}
where $F_{\log}$ is given in \cite[Example~3.6]{BMcS06} and $\overline{\mathcal{D}^{F_{\log}}}$ is in \eqref{fitz:dist}. We have that
\begin{align*}
\bprox{\gamma,\theta}^{\dagger F_{\log}}(x) &= \begin{cases}
\left(\frac{1}{\gamma}+1\right)e^{\gamma}\gamma x & \text{if~}  x\leq \frac{e^{-\gamma}}{2+2\gamma};\\
\left(\frac{1}{\gamma}-1\right)e^{-\gamma}\gamma x & \text{if~}  \frac{e^\gamma}{2-2\gamma} < x;\\
\frac{1}{2} & \text{otherwise},
\end{cases}\\
\text{and}\quad \benv[\dagger F_{\log}]{\gamma,\theta}(x)&=\begin{cases}
\frac{1}{2}-e^{\gamma}x & \text{if~}  x\leq \frac{e^{-\gamma}}{2+\gamma};\\
\gamma \left( \left(\frac{1}{\gamma}-1 \right)e^{-\gamma}x +e^\gamma x  \right) -\frac{1}{2} & \text{if~}  x> \frac{e^\gamma}{2-\gamma};\\
\frac{\left(\W\left(\frac{e}{2x} \right)-1 \right)^2}{2\gamma \W \left(\frac{e}{2x} \right)} & \text{otherwise}.
\end{cases}
\end{align*}
The envelope is shown in Figure~\ref{fig:entropy_left_fitz}.
\end{example}

\begin{example}[Right prox and envelope for $\overline{\mathcal{D}}^{F_{\log}}$]

We next consider the right envelope and corresponding proximity operator characterized by 
\begin{align*}
\fenv[\dagger F_{\log}]{\gamma,\theta}(x)&= \underset{y \in \RR_+}{\inf} \{\theta(y)+\frac{1}{\gamma}\overline{\mathcal{D}}^{F_{\log}}(x,y)  \}\\
&=\theta(\fproxs{\gamma,\theta}^{\dagger F_{\log}} (x))+\frac{1}{\gamma}\overline{\mathcal{D}}^{F_{\log}}\left(x,\fproxs{\gamma,\theta}^{\dagger F_{\log}} (x)\right),\\
\text{where}\quad \fproxs{\gamma,\theta}^{\dagger F_{\log}} (x) &\in \fprox{\gamma,\theta}^{\dagger F_{\log}} (x).
\end{align*}
We use the selection operator $\fproxs{\gamma,\theta}^{F_{\log}}$ because, in this case, the prox operator $\fprox{\gamma,\theta}^{F_{\log}}$ is set-valued:
\begin{align*}
\fprox{\gamma,\theta}^{\dagger F_{\log}}(x) = \begin{cases}
0 & \text{if~}  x=0 \;\text{and} \; e<\frac{1}{\gamma};\\
[0,\frac{1}{2}] & \text{if~}  x=0 \;\text{and}\; \frac{1}{\gamma} = e;\\
\frac{1}{2} & \text{if~}  x=0 \;\text{and}\; \frac{1}{\gamma}< e;\\
-\frac{\gamma x \W \left( -\frac{1}{\gamma} \right)}{\W \left( -\frac{1}{\gamma} \right)+1} & \text{if~}  0<x<-\frac{\gamma (\W(-\gamma)+1) }{2\W(-\gamma)} \; \text{and} \; e<\frac{1}{\gamma};\\
\frac{\gamma x \W \left(\frac{1}{\gamma} \right)}{\W \left(\frac{1}{\gamma} \right)+1} & \text{if~}  \frac{\gamma (\W(\gamma)+1) }{2\W(\gamma)}<x;\\
\frac{1}{2} & \text{otherwise}. 
\end{cases}
\end{align*}
The corresponding right envelope is
\begin{align*}
\fenv[\dagger F_{\log}]{\gamma,\theta}(x)=& \begin{cases}
\frac{1}{2} & \text{if~}  x=0 \;\text{and}\;e \leq \frac{1}{\gamma};\\
\frac{1}{2\gamma e} & \text{if~}  x=0 \;\text{and}\; \frac{1}{\gamma} < e;\\
\frac{x\W(\gamma)}{\gamma(\W(\gamma)+1)}-\frac{1}{2}+\frac{x \left(\W \left( \frac{e\gamma(\W(\gamma)+1)}{\W(\gamma)} \right) -1 \right)^2}{\gamma \W \left( \frac{e\gamma(\W(\gamma)+1)}{\W(\gamma)} \right)} & \text{if~}  \frac{\gamma(\W(\gamma)+1)}{2\W(\gamma)}<x;\\
\frac{x\W(-\gamma)}{\gamma(\W(-\gamma)+1)}+\frac{1}{2}+\frac{x \left(\W \left(- \frac{e\gamma(\W(-\gamma)+1)}{\W(-\gamma)} \right) -1 \right)^2}{\gamma \W \left(- \frac{e\gamma(\W(-\gamma)+1)}{\W(-\gamma)} \right)} & \text{if~}  0<x<-\frac{\gamma(\W(-\gamma)+1)}{2\W(-\gamma)}\;\text{and}\; e<\frac{1}{\gamma};\\
\frac{x(\W(2xe)-1)^2}{\gamma \W(2xe)} & \; \text{otherwise}.
\end{cases}
\end{align*}
The envelope is shown in Figure~\ref{fig:entropy_right_fitz}.
\end{example}

\begin{example}[Left prox and envelope for {$\overline{\mathcal{D}}^{\sigma_{\log}}$}]

Turning to the biggest of the representative functions for the logarithm, we next consider the left envelope and corresponding proximity operator characterized by
\begin{align*}
\benv[\dagger \sigma_{\log}]{\gamma,\theta}(x)&= \underset{y \in \RR_+}{\inf} \{\theta(y)+\frac{1}{\gamma}\overline{\mathcal{D}}^{\sigma_{\log}}(y,x)  \}\\
&=\theta(\bprox{\gamma,\theta}^{\dagger \sigma_{\log}} (x))+\frac{1}{\gamma}\overline{\mathcal{D}}^{\sigma_{\log}}(\bprox{\gamma,\theta}^{\dagger \sigma_{\log}} (x),x).
\end{align*}
where $\overline{\mathcal{D}^{\sigma_{\log}}}$ is as in \ref{eqn:Dsigma}. We have that
\begin{equation*}
\bprox{\gamma,\theta}^{\dagger \sigma_{\log}} (x)= \begin{cases}
\frac{x}{e^{-\gamma+1}} & \text{if~} x \leq \frac{1}{2}e^{-\gamma+1}\;\text{and}\; 1\leq \gamma;\\
x & \text{if~}  x \geq \frac12 \;\text{or}\; \gamma<1;\\
\frac12 & \text{otherwise}.
\end{cases}
\end{equation*}
The corresponding envelope is
\begin{equation*}
\benv[\sigma_{\log}]{\gamma,\theta}(x) = \begin{cases}
0 & \text{if~}  x=0;\\
x-\frac12 & \text{if~}  x \geq \frac12;\\
\frac12 -x & \text{if~}  x < \frac12 \;\text{and}\; \gamma < 1;\\
-\frac{x}{\gamma}e^{\gamma-1}\left(-\log \left(xe^{\gamma-1} \right)+\gamma+\log(x) \right) +\frac12 & \text{if~} x \leq \frac12 e^{-\gamma+1}\;\text{and}\; 1 \leq \gamma ;\\
-\frac{1}{2\gamma}\log(2x) & \text{otherwise}.
\end{cases}.
\end{equation*}
The envelope is shown in Figure~\ref{fig:entropy_left_sigma}.
\end{example}

\begin{example}[Right prox and envelope for {$\overline{\mathcal{D}}^{\sigma_{\log}}$}]
We consider the right envelope and corresponding proximity operator characterized by 
\begin{align*}
\fenv[\dagger \sigma_{\log}]{\gamma, \theta}(x)&= \underset{y \in \RR_+}{\inf} \{\theta(y)+\frac{1}{\gamma}\overline{\mathcal{D}}^{\sigma_{\log}}(x,y)  \}\\
&=\theta(\fprox{\gamma,\theta}^{\dagger \sigma_{\log}} (x))+\frac{1}{\gamma}\overline{\mathcal{D}}^{\sigma_{\log}}\left(x,\fprox{\gamma,\theta}^{\dagger \sigma_{\log}} (x)\right).
\end{align*}
The corresponding proximity operator is
\begin{equation*}
\fprox{\gamma,\theta}^{\dagger \sigma_{\log}} (x)=\begin{cases}
\frac12 & \text{if~}  \frac12 < x \;\text{and}\; x\leq \frac12\gamma ;\\
\frac{x}{\gamma} & \text{if~}  1 \leq \gamma \;\text{and}\; \frac12 \gamma < x ;\\
x &\text{otherwise},
\end{cases}
\end{equation*}
while the corresponding envelope is
\begin{equation*}
\fenv[\dagger \sigma_{\log}]{\gamma, \theta}(x) = \begin{cases}
\frac12-x & \text{if~}  x \leq \frac12;\\
x-\frac12 & \text{if~}  \frac12 < x \;\text{and}\; \gamma \leq 1;\\
-\frac{1}{2}+\frac{x}{\gamma}(1+\log(\gamma)) & \text{if~}  1\leq \gamma \;\text{and}\; \frac{\gamma}{2} < x ;\\
\frac{x}{\gamma}\log(2x) & \text{otherwise}.
\end{cases}
\end{equation*}
The envelope is shown in Figure~\ref{fig:entropy_right_sigma}.
\end{example}

The operators in Example~\ref{ex:bregman} (and their computation) may be found in \cite{BDL17}, with the minor modification that here we are computing with the lower closure of the distance and so obtain closed forms which differ at zero. We include them here for their comparison with the new GBDs for the logarithm.

\begin{example}[Proxes and envelopes for $\overline{\mathcal{D}}^{\ent \oplus \ent^*}$]\label{ex:bregman}
We next consider the case when $\overline{\mathcal{D}}^{\ent \oplus \ent^*}$ is the closed GBD for the Fenchel--Young representative of $\log$ \eqref{def:Bregman_lower_closure}, a case whose relationship to the Bregman distance for the Boltzmann--Shannon entropy is discussed in Example~\ref{ex:KullbackLiebler}. The left proximity operator and envelope are given by
\begin{subequations}
\begin{align*}
\bprox{\gamma,\theta}^{\dagger \ent \oplus \ent^*}(y) &=\begin{cases}
y\exp(\gamma), &\text{~if~} 0 \leq y <\tfrac{1}{2}\exp(-\gamma);
\\[+2mm]
y\exp(-\gamma), &\text{~if~} y >\tfrac{1}{2}\exp(\gamma);
\\[+2mm]
\tfrac{1}{2}, &\text{~otherwise},
\end{cases}\\[+4mm]
\benv[\dagger \ent \oplus \ent^*]{\gamma,\theta}(y) &=\begin{cases}
\tfrac{y(1-e^\gamma)}{\gamma} + \tfrac{1}{2}, &\text{~if~} 0 \leq y
<\tfrac{1}{2}\exp(-\gamma); \\[+2mm]
\tfrac{y(1-e^{-\gamma})}{\gamma} - \tfrac{1}{2}, &\text{~if~}
\tfrac{1}{2}\exp(\gamma) <y; \\[+2mm]
\tfrac{2y- \ln(y) -1 -\ln(2)}{2\gamma}, &\text{~otherwise}.
\end{cases}
\end{align*}
\end{subequations}
For details, see \cite[Example~4.1(ii)]{BDL17}. The envelope is shown in Figure~\ref{fig:entropy_left_bregman}.
The right proximity operator and envelope are given by
\begin{align*}
\fprox{\gamma,\theta}^{\dagger \ent \oplus \ent^*}(x) &=\begin{cases}
\tfrac{x}{1 -\gamma}, &\text{~if~} 0 \leq x <\tfrac{1 -\gamma}{2};
\\[+2mm]
\tfrac{x}{1 +\gamma}, &\text{~if~} x >\tfrac{1 +\gamma}{2};
\\[+2mm]
\tfrac{1}{2}, &\text{~otherwise}.
\end{cases}\\
\fenv[\dagger \ent \oplus \ent^*]{\gamma, \theta}(x) &=\begin{cases}
\tfrac{\ln(1 -\gamma)}{\gamma}x + \tfrac{1}{2}, &\text{~if~} 0 \leq x
<\tfrac{1 -\gamma}{2}; \\[+2mm]
\tfrac{\ln(1 +\gamma)}{\gamma}x - \tfrac{1}{2}, &\text{~if~} x
>\tfrac{1 +\gamma}{2}; \\[+2mm]
\tfrac{1}{\gamma}\left(x\ln(2x) -x +\tfrac{1}{2}\right), &\text{~otherwise}.
\end{cases}
\end{align*}
For details, see \cite[Example~4.1(ii)]{BDL17}. The envelope is shown in Figure~\ref{fig:entropy_right_bregman}.
\end{example}

\end{document}